\documentclass[12pt]{amsart}
\usepackage[top=1in,bottom=1in, left=1 in, right= 1 in, includehead,includefoot]{geometry}

\usepackage{amsthm}
\usepackage{amsmath,amssymb}
\usepackage[T1]{fontenc}
\usepackage{dsfont}
\usepackage{color,hyperref}
\usepackage[numbers,sort&compress]{natbib}
\usepackage{enumerate}
\usepackage[utf8]{inputenc}

\newtheorem{theorem}{Theorem}[section]

\newtheorem{lemma}[theorem]{Lemma}
\newtheorem{proposition}[theorem]{Proposition}

\theoremstyle{definition}

\theoremstyle{remark}
\newtheorem{remark}[theorem]{Remark}
\theoremstyle{remark}

\numberwithin{equation}{section}

\def\fddto{\xrightarrow{\textit{f.d.d.}}}
\newcommand{\ind}{{\bf 1}}

\def\inddd#1{{\ind}_{\left\{#1\right\}}} 
\newcommand{\proba}{\mathbb P}
\newcommand{\esp}{{\mathbb E}}

\newcommand{\inv}{^{-1}}
\newcommand{\cov}{{\rm{Cov}}}

\newcommand{\eqnh}{\begin{eqnarray*}}
\newcommand{\eqne}{\end{eqnarray*}}
\newcommand{\eqnhn}{\begin{eqnarray}}
\newcommand{\eqnen}{\end{eqnarray}}
\newcommand{\equh}{\begin{equation}}
\newcommand{\eque}{\end{equation}}

\def\summ#1#2#3{\sum_{#1 = #2}^{#3}}

\def\sif#1#2{\sum_{#1=#2}^\infty}

\newcommand{\eqd}{\stackrel{d}{=}}

\def\topp#1{^{(#1)}}

\def\abs#1{\left|#1\right|}

\def\ccbb#1{\left\{#1\right\}} 

\def\pp#1{\left(#1\right)}
\def\spp#1{(#1)}

\def\floor#1{\left\lfloor #1 \right\rfloor}
\def\sfloor#1{\lfloor #1 \rfloor}
\def\ceil#1{\left\lceil #1 \right\rceil}

\def\vv#1{{\boldsymbol #1}}

\def\qmand{\quad\mbox{ and }\quad}

\def\qmwith{\quad\mbox{ with }\quad}
\def\mfa{\mbox{ for all }}

\def\mmas{\mbox{ as }}

\def\wt#1{\widetilde{#1}}
\def\wb#1{\overline{#1}}
\def\what#1{\widehat{#1}}

\def\limn{\lim_{n\to\infty}}

\def\limsupn{\limsup_{n\to\infty}}

\def\weakto{\Rightarrow}

\def\Z{{\mathbb Z}}

\def\R{{\mathbb R}}

\def\N{{\mathbb N}} 
\def\BB{{\mathbb B}}

\newcommand{\calP}{{\mathcal P}}

\newcommand{\calM}{\mathcal{M}}
\newcommand{\calZ}{\mathcal{Z}}

\renewcommand{\d}{{\rm d}}
\newcommand{\aswto}{\stackrel{a.s.w.}\to}

\date{}
\def\Var{\mathop{\rm Var}}
\def\Cov{\mathop{\rm Cov}}

\def\R{\mathbb{R}}
\def\C{\mathbb{C}}
\def\N{\mathbb{N}}
\def\P{\mathbb{P}}
\def\E{\mathbb{E}}
\renewcommand{\i}{{\rm i}}

\title[A functional central limit theorem for Karlin model]{A functional central limit theorem for weighted occupancy processes of the Karlin model}

\author{Jaime Garza}
\address
{
Department of Mathematical Sciences\\
University of Cincinnati\\
2815 Commons Way\\
Cincinnati, OH, 45221-0025, USA.
}
\email{garzaje@mail.uc.edu}

\author{Yizao Wang}
\address
{
Department of Mathematical Sciences\\
University of Cincinnati\\
2815 Commons Way\\
Cincinnati, OH, 45221-0025, USA.
}
\email{yizao.wang@uc.edu}

\begin{document}\sloppy
\begin{abstract}
A functional central limit theorem is established for weighted occupancy processes of the Karlin model. The weighted occupancy processes take the form of, with $D_{n,j}$ denoting the number of urns with $j$-balls after the first $n$ samplings, $\sum_{j=1}^na_jD_{n,j}$ for a prescribed sequence of real numbers $(a_j)_{j\in\N}$. The main applications are limit theorems for random permutations induced by Chinese restaurant processes with $(\alpha,\theta)$-seating with $\alpha\in(0,1), \theta>-\alpha$. An example is briefly mentioned here, and full details  are provided in an accompanying paper.\end{abstract}

\maketitle

\section{Introduction}
The infinite urn scheme has been a classical model in probability theory, dating back at least to the 1960s in the works of \citet{bahadur60number}, \citet{darling67some} and \citet{karlin67central}. 
In particular, Darling and Karlin first considered frequencies decaying polynomially, that is, the sampling frequencies $(p_j)_{j\in\N}$ of the infinite urn model satisfy $p_j = L(j) j^{-1/\alpha}$ for some slowly varying function $L$ at infinity and some parameter $\alpha\in(0,1]$. They also examined the case when $p_j$ decays faster than any polynomial rates 
(for a precise definition, see \eqref{eq:RV} below with $\alpha=0$).
For these cases,
Karlin established central limit theorems for various statistics, 
and therefore the model with such frequencies
 has been sometimes referred to as the Karlin model (although his seminal paper investigated other frequencies as well). 

Among many statistics of interest, Karlin established central limit theorems for the so-called occupancy and odd-occupancy processes. 
 If we let $D_{n,j}$ denote the number of urns with exactly $j$ balls after the first $n$ rounds of sampling, then these are the counting processes
\equh\label{eq:Karlin}
S_n:=\summ j1n D_{n,j} \qmand S_n^{\rm odd}:=\summ j1n D_{n,j}\inddd{j \rm~is~odd}, n\in\N,
\eque
respectively.
 See \citet{gnedin07notes} for a survey on the infinite urn schemes.
Surprisingly, the functional central limit theorems for the Karlin model have not been developed until the last 10 years, even Karlin's results already indicated non-diffusive scaling limits for occupancy and odd-occupancy processes more than half a century ago. For functional central limit theorems on statistics of Karlin model and its variations, we mention \citep{durieu16infinite,chebunin16functional,chebunin22functional,durieu20infinite,iksanov22functional,iksanov22small}. 
In particular, in \citep{durieu16infinite} functional central limit theorems for occupancy and odd-occupancy processes have been established. Moreover, an additional randomization of the Karlin model was introduced and it was shown that the randomized odd-occupancy process scales to a fractional Brownian motion with Hurst index $H = \alpha/2\in(0,1/2)$. This result revealed a connection of the odd-occupancy processes to stochastic processes with long-range dependence \citep{pipiras17long,samorodnitsky16stochastic}.

In the current paper,
we 
establish a functional central limit theorem for
 the weighted sum of occupancy counts $(D_{n,j})_{j\in\N}$ with prescribed weights $\vec a = (a_j)_{j\in\N}$ defined as
\[
S_n\equiv S_{n,\vec a}:= \summ j1n a_j D_{n,j}.
\]
The statistics in this form include the (odd-)occupancy processes in \eqref{eq:Karlin} and hence we refer to $S_n$ as {\em weighted occupancy processes}. Throughout, $a_0 = 0$.  
In particular, with $a_j = 1$ and $a_j = \inddd{j~\rm~is~odd}$, $S_{n,\vec a}$ becomes the two cases in \eqref{eq:Karlin} 
considered by Karlin respectively; 
see also \citep{durieu16infinite}.
 Another extensively studied case is $a_j = \inddd{j\ge j_0}$, and $S_{n,\vec a} =  \sum_{j\ge j_0}D_{n,j}$ is the counting process of all urns with at least $j_0$ balls. This case was again investigated already by Karlin; see also \citep{chebunin16functional}. 
As an immediate consequence one obtains limit theorems for $D_{n,j_0}$ by simply writing $D_{n,j_0} = \sum_{j\ge j_0}D_{n,j} - \sum_{j\ge j_0+1}D_{n,j}$ and a continuous mapping argument. We are not aware of any other examples of $S_{n,\vec a}$ investigated in the literature.

For the general weights $(a_j)_{j\in\N}$ we impose $|a_{i+j}-a_i|\le Cj^\beta$ for some $\beta\in[0,\alpha^2/2)$ and 
all
$i,j\in \N$. When $\beta = 0$ this implies that the weights are bounded, including the occupancy and odd-occupancy processes investigated earlier.  Our main result is a functional central limit theorem for the counting process $(S_{\floor{nt}})_{t\in[0,1]}$ appropriately normalized, and we identify a new family of Gaussian processes in the limit. Our results can be viewed as an extension of the functional central limit theorems in \citep{durieu16infinite} (more precisely, Theorem 2.3 therein for the {\em quenched} version of the randomized model with fixed randomization $\varepsilon_j\equiv 1$, corresponding exactly the two cases in \eqref{eq:Karlin} here; see Remark \ref{rem:explicit} for more details). While we follow the standard Poissonization approach, it turned out that the analysis becomes much more involved even in the case where $(a_j)_{j\in\N}$ is a bounded sequence. The technical challenges came from (a) proving the tightness for the Poissonized process where, due the fact that now $a_j$ may be unbounded (when $\beta>0$) we no longer have simple moment estimates, and (b) carrying out of the de-Possonization step, where we use a refined analysis on discrete local central limit theorems for Poisson and binomial distributions.

Our main motivation came from random matrix theory: using the general functional central limit theorem established here we shall investigate asymptotic behaviors of random permutation matrices induced by the Chinese restaurant processes with $(\alpha,\theta)$-seating. We briefly mention one application in Section \ref{sec:RMT}: a functional central limit theorem for the linear statistics of eigenvalues of the aforementioned permutation matrices. 
In the accompanying paper \citep{garza24limit} we provide full details and also another application on limit fluctuations of characteristic polynomials of the permutation matrices outside the support of its limit  (which is the uniform measure on the unit circle $\{z\in \C:|z| = 1\}$).

\subsection{Main results}
 We first recall the setup of an infinite urn model with sampling frequencies $(p_j)_{j\in\N}$, where $p_1\ge p_2\ge\cdots$ and $\sif j1 p_j = 1$. Each round, $Y_i$ is an independent sampling with $\proba(Y_i = j) = p_j$ indicating the label of the urn that a ball is thrown in, and we consider the following statistics at each time $n\in\N$:
\[
K_{n,\ell} := \summ i1n \inddd{Y_i = \ell}, \quad \ell\in\N \qmand D_{n,j}:=\sif \ell1 \inddd{K_{n,\ell} = j}, \quad j\in\N. 
\]
In words, $K_{n,\ell}$ is the number of balls in the urn labeled by $\ell$ by round $n$ and $D_{n,j}$ is the number of urns with exactly $j$ balls by round $n$. The key here is the decay rate of the frequencies. 
  Let $0<\alpha<1$. The following assumption was introduced first by \citet[Section 4]{darling67some} and \citet{karlin67central} about the same time (Darling only investigated weak laws of large number and Karlin investigated central limit theorems)
\begin{equation}\label{eq:RV}
    v(x):=\sum_{i=1}^\infty \delta_{1/p_i}([0,x])=x^\alpha L(x), \quad x\ge0,
\end{equation}
where $L$ is a slowly varying function at infinity.
We note that if 
\equh\label{eq:p RV}
p_n\sim \mathsf C_0n^{-1/\alpha},
\eque as $n\to\infty$ for some constant $\mathsf C_0>0$, 
 then $v(x)\sim \mathsf C_0^{\alpha}x^\alpha$ 
 as $x\rightarrow \infty$. Throughout we write $a_n\sim b_n$ if $\limn a_n/b_n = 1$. More generally, one can consider $p_n = n^{-1/\alpha}L(n)$ for some slowly varying function $L$ at infinity. However we do not pursue this generality, as for our main applications on 
 random matrices, $L(n)$ converges to a constant almost surely, and hence \eqref{eq:p RV} suffices. 
 The proofs are already quite involved without considering a more general $L$.

We next introduce the limit process that will arise in our limit theorem. 
Let $\calM_\alpha$ be a Gaussian random measure defined on $\R_+\times \Omega'$ with intensity measure $\alpha r^{-\alpha-1}\d r\d \P'$ and $(N'(t))_{t\geq0}$ be a standard Poisson process on $(\Omega',\P')$. 
     Set
\equh\label{eq:calZ}
\calZ_{\alpha}(t)\equiv \calZ_{\alpha,\vec a}(t):=\int_{\R_+\times\Omega'} \pp{a_{N'(rt)}-\esp'a_{N'(rt)}}\calM_\alpha(\d r,\d\omega'), t\ge 0,
\eque
for a certain sequence of real numbers $\vec a= (a_j)_{j\in\N_0}$ with $\N_0 = \{0,1,\dots\}$ and $a_0 = 0$. Note that the integrand can also be written as 
\[
a_{N'(rt)} - \esp a_{N'(rt)} = \sif j1 a_j(\inddd{N'(rt) = j} - \proba'(N'(rt) = j)).
\] 
We refer to \citep{samorodnitsky94stable} for background on stochastic integrals with respect to Gaussian random measures, and will comment on the conditions imposed on $(a_j)_{j\in\N}$ in a moment. We only recall the facts that if $\calM$ is a Gaussian random measure on a measurable space $(E,\mathcal E)$ with control measure $\mu$, then $\int f\d \calM$ is a centered Gaussian random variable for $f\in L^2(E,\mu)$ and $\cov(\int f\d\calM,\int g\d\calM) = \int fg\d\mu$ with $f,g\in L^2(E,\mu)$.

For our functional central limit theorem, set 
\equh\label{eq:W_n}
W_n(t):=\summ j1{\floor{nt}} a_j\pp{D_{\floor{nt},j}-\esp D_{\floor{nt},j}}, t\in[0,1], n\in\N.
\eque
Write also $\sigma_n := v(n)^{1/2}, n\in\N$. 
Under \eqref{eq:p RV}, we have $\sigma_n\sim \mathsf C_0^{\alpha/2} n^{\alpha/2}$. 
Throughout, the constant $C$ may change from line to line. 
\begin{theorem}\label{thm:2}
Assume that $(p_j)_{j\in\N}$ satisfies \eqref{eq:p RV} with $\alpha\in(0,1)$. Assume that there exists $\beta\in[0,\alpha^2/2)$ such that $(a_j)_{j\in\N}$ satisfies that 
\equh\label{eq:a increment}
|a_{i+j}-a_i|\le Cj^\beta,\quad \mfa i,j\in\Z_{\ge1}.
\eque
   Then,
    \[
\pp{ \frac{{W}_n(t)}{\sigma_n}}_{t\in[0,1]}\weakto \pp{\calZ_{\alpha}(t)}_{t\in[0,1]}, \mmas n\rightarrow \infty,
\]
    in  the space of c\`adl\`ag functions $D([0,1])$  with $J_1$ topology \citep{billingsley99convergence},  where the process $\calZ_{\alpha}$ is as in \eqref{eq:calZ}.
\end{theorem}
\begin{remark}
We first comment on condition \eqref{eq:a increment} on $(a_j)_{j\in\N}$.
\begin{enumerate}[(i)]
\item
The condition \eqref{eq:a increment} implies  $|a_j|\le C j^\beta$, and this weaker condition alone  with $\beta<\alpha/2$ is needed for $(\calZ_\alpha(t))_{t\ge 0}$ to be a family of well-defined Gaussian random variables. See Proposition \ref{prop:Z}. If say $a_j\sim Cj^\beta$ with $\beta>\alpha/2$ we expect a phase transition: a different normalization should be applied and a different stochastic process arises. This is left for future work.
\item 
We impose condition \eqref{eq:a increment} with $\beta < \alpha/2$ to show that the process $(\calZ_\alpha(t))_{t\ge0}$ has a version with H\"older-continuous sample paths of index $\gamma < \alpha/2$.
See again Proposition~\ref{prop:Z}. 
\item Condition \eqref{eq:a increment} with the more restrictive constraint $\beta<\alpha^2/2$ (recall $\alpha\in(0,1)$) is needed to show that the Poissonized model is tight in $D([0,1])$, which is necessary to show the functional central limit theorem for the Poissonized version in Section \ref{sec:tightness}. Assuming the weaker condition $|a_j|\le Cj^\beta$ with $\beta<\alpha/2$ and no condition on the increments, we managed to prove the convergence of finite-dimensional distributions for the Poissonized model in Proposition \ref{prop:fdd}; although this result alone is not enough to prove Theorem \ref{thm:2} (in fact, even the central limit theorem with $t=1$ fixed cannot be obtained). Note that we also need \eqref{eq:a increment} with $\beta<\alpha/2$ in the de-Poissonization step. It is plausible that the constraint $\beta<\alpha^2/2$ can be relaxed to $\beta<\alpha/2$ in Section \ref{sec:tightness}.
\item 
It is not clear to us  when $a_j$ is unbounded whether the functional central limit theorem can be established without any further conditions beside $|a_j|\le Cj^\beta$. It is known that when $a_j$ changes sign, the regularity of the sample paths can be quite poor. For a closely related issue, see \citep[Remark 2]{durieu20infinite}. 
\end{enumerate}    \end{remark}

      \begin{remark}
For each $j\in\N$, set 
\[
         \calZ_{\alpha,j}(t):=\int_{\R_{+}\times \Omega'}(\inddd{N'(rt)=j}-\P'(N'(rt)=j))\calM_\alpha(\d r,\d \omega'), t\ge 0.
\]
Our process $\calZ_\alpha$ is in fact a linear combination of the above with coefficients $(a_j)_{j\in\N}$. That is, 
\[
\sif j1a_j\calZ_{\alpha,j}(t) = \int_{\R_+\times\Omega_+}\sif j1a_j(\inddd{N'(rt)=j}-\P'(N'(rt)=j))\calM_\alpha(\d r,\d \omega'), t\ge 0,
\]
where the left-hand side is understood as the $L^2$-limit of the infinite series and the right-hand side is a stochastic integral. 

The processes $\calZ_{\alpha,j}$ and their linear combinations have appeared in earlier investigations of Karlin models. 
It is worth noting in the literature, limiting Gaussian processes are often characterized by their explicit covariance functions, with no stochastic-integral representations provided. However, stochastic-integral representations offer additional insight into Gaussian processes.
In fact, in \citep{durieu16infinite}, the stochastic-integral representations can only be read implicitly in the proofs, and they are explained explicitly only later in \citep{fu20stable}.
\end{remark}
\begin{remark}
We mention the recent paper by \citet{iksanov22small} (see also \citep[Proposition 4.1]{iksanov22functional}), who investigated Karlin model with $\alpha=0$ in \eqref{eq:RV} (note that in this case \eqref{eq:p RV} does not make sense; but instead $p_j$ decays faster than any polynomial rates), and obtained new Gaussian processes as the scaling limits for various statistics. In short, the case $\alpha=0$ of the Karlin model behaves drastically different from the case $\alpha\in(0,1)$ investigated here, yet it appears to have received much less attention. The Karlin model with $\alpha=0$ is closely related to the Bernoulli sieves \citep{gnedin20nested,iksanov16renewal}. See also \citep{iksanov22small,iksanov22functional} for related references. 
\end{remark}

\subsection{An application in random matrix theory}
\label{sec:RMT}
The key connection to random permutation matrices is that, thanks to Kingman's representation theorem \citep{kingman82coalescent}, the occupancy counts  $(D_{n,1},\dots,D_{n,n})$ of an infinite urn scheme have the same joint law of the cycle counts $(C_{n,1},\dots,C_{n,n})$ of the corresponding exchangeable random partition from a Chinese restaurant process (where $C_{n,j}$ is the number of the $j$-cycles of the partition of $[n]=\{1,2,...,n\}$), if the rate $(p_j)_{j\in\N}$ is appropriately chosen (strictly speaking, this statement involves regular conditional probability; see \citep{garza24limit} for more details). Since for a permutation the cycle structure determines completely the eigenvalues of the corresponding matrix, many statistics of interest can be expressed as weighted sum of cycle counts, which have the same law as the weighted occupancy counts investigated here. This connection allows us to study random permutation matrices by 
investigating classical questions on infinite urn schemes. This idea was first due to \citet{diaconis94eigenvalues} and has been applied extensively to random permutation matrices either uniform or those induced by Ewens measures on symmetric groups with parameter $\theta>0$, corresponding to the Chinese restaurant processes with $(0,\theta)$-seating (see e.g.~\citep{arratia03logarithmic,bahier19characteristic,benarous15fluctuations,wieand00eigenvalue,bahier22smooth}).
At the same time, the Karlin model with parameter $\alpha\in(0,1)$ can be associated to Chinese restaurant processes with $(\alpha,\theta)$-seating ($\theta>-\alpha$), and this indicates again the importance of Karlin model from another aspect.  Surprisingly, we have found very few results on random permutations induced by $(\alpha,\theta)$-seating with $\alpha\in(0,1)$, and the connection to the infinite urn model has not been fully exploited yet.  

Here we mention briefly an application in our accompanying paper \citep{garza24limit}, just to provide a flavor of the results to be established. We skip most 
of the background including the definition of Chinese restaurant processes, as more details and other applications can be found therein. 

A Chinese restaurant process consists of a sequence of exchangeable random partitions, $(\Pi_n)_{n\in\N}$, each is a random partition of $[n]$. Assuming further a uniformly random ordering of each block of the partition leads to a random permutation of $[n]$, we let $M_n$ denote the associated permutation matrix.  Our goal is to study the asymptotic behavior of eigenvalues of $M_n$ as $n\to\infty$. 
Namely for each $M_n$, letting 
$e^{\i2\pi\theta_{n,1}},\dots,e^{\i2\pi\theta_{n,n}}$ 
represent its $n$ eigenvalues, with $\theta_{n,j}\in[0,1),j=1,\dots,n$, 
we are interested in the asymptotic behavior of the following linear statistics of the spectrum of a matrix $M_n$ for some test function $f:[0,1)\to \R$:
\[
S_n(f):=\summ i1n f(\theta_{n,i}).
\]

Before stating the limit theorem, a key notion that needs to be recalled is the asymptotic frequencies of the random partitions $(\Pi_n)_{n\in\N}$, which, ordered in decreasing order are denoted by 
$(P_j^\downarrow)_{j\in\N}$. (In words, $P_j$ without ordering is the asymptotic proportion of customers sitting at the $j$-th table.)  It is well-known that they are non-deterministic, strictly positive, and satisfy $\sif i1 P_i^\downarrow = 1$, and in particular
 for $\alpha\in(0,1), \theta>-\alpha$, 
\[
P_n^\downarrow\sim \pp{\frac{\mathsf s_{\alpha,\theta}}{\Gamma(1-\alpha)}}^{1/\alpha}n^{-1/\alpha},
\quad \mbox{ almost surely.}
\]
Here and throughout, $\Gamma$ is the gamma function. 
The random variable  $\mathsf s_{\alpha,\theta}$ referred to as the {\em $\alpha$-diversity}, is strictly positive and is measurable with respect to 
\[
\calP := \sigma\pp{(P_i^\downarrow)_{i\in\N}}.
\]

The limit theorems on random matrices are stated in the notion of almost sure weak convergence. 
Given a Polish space $(\mathcal X,d)$ and random elements $(X_n)_{n\in\N}, X$ in it, we say $X_n$ converges almost surely weakly to $X$ as $n\to\infty$ with respect to $\calP$, denoted by $X_n\aswto X$ with respect to $\calP$, if for all continuous and bounded functions $g$ on $\mathcal X$, $\limn \esp \pp{g(X_n)\mid \calP} = \esp \pp{g(X)\mid\calP}$ almost surely.
For a slightly more general notion, see \citep{grubel16functional}.

Introduce
\[
a_j(f):=jR_j(f) \qmwith R_j(f) := \frac1j \summ k0{j-1}f\pp{\frac kj} - \int_0^1f(x)\d x, j\in\N.
\]

\begin{theorem}
For functions $f:[0,1)\to \R$ such that $(a_j)_{j\in\N}\equiv (a_j(f))_{j\in\N}$ satisfy the assumptions in Theorem \ref{thm:2}, 
we have
\[
\pp{\frac{S_{\floor {nt}}(f) - \esp \pp{S_{\floor{nt}}(f)\mid\calP}}{n^{\alpha/2}}}_{t\in[0,1]}\aswto \pp{\frac{\mathsf s_{\alpha,\theta}}{\Gamma(1-\alpha)}}^{1/2}\pp{\calZ_{\alpha}(t)}_{t\in[0,1]}
\]
with respect to $\calP$ in $D([0,1])$, where on the right-hand side the process $\calZ_{\alpha}$ is as in \eqref{eq:calZ} and independent from $\calP$. 
\end{theorem}
The class of functions satisfying $|a_j(f)|\le C$ include for example indicator functions $f = \ind_{[a,b)}$ and the class $W^{1,p}([0,1])$ with $p\in[1,\infty]$ (absolutely continuous functions with $f'\in L^p([0,1])$). 

Our results are to be compared with limit theorems for linear statistics of random permutation matrices following Ewens measures by \citet{benarous15fluctuations}. There are two noticeable differences. First, they assumed $\sif j1 a_j^2/j =\infty$ and $\max_{j=1,\dots,n}|a_j| = o(s_n)$ with $s_n = (\summ j1n a_j^2/j)^{1/2}$ (guaranteed by the assumption that $f$ is of bounded variation), and in this case it was shown that the normalization for the central limit theorem is $s_n$ which is of order at most $O(\sqrt{\log n})$ (this is the Gaussian regime, and when $\sif j1 a_j^2/j<\infty$ the limit is a weighted sum of centered Poisson random variables). Here, the normalization is $n^{\alpha/2}$. Second, they only established a central limit theorem (for $t=1$ fixed), and their method based on Feller's coupling does not allow one to establish a functional central limit theorem as ours, which provide additional information on the temporal evolution of eigenvalues. (Instead, the Feller's coupling allows one to establish another type of functional central limit theorem as in \citet{hansen90functional}.)

The rest of the paper is organized as follows. Section \ref{sec:Z} provides more properties on the limit process $\calZ$.  Section \ref{sec:Karlin} provides the proof for the Poissonized model. Section \ref{sec:dePoissonization} provides the proof of the de-Poissonization step.

\subsection*{Acknowledgements}
The authors would like to thank two anonymous referees for very careful reading of and helpful comments on the manuscript.
Y.W.~was grateful to Alexander Iksanov for several illuminating discussions. J.G.~and Y.W.~were partially supported by Army Research Office, US (W911NF-20-1-0139).
 Y.W.~was also partially supported by Simons Foundation (MP-TSM-00002359), and a Taft Center Fellowship (2024--2025) from Taft Research Center at University of Cincinnati.

\section{The limit process}\label{sec:Z}
Recall the process $\calZ_\alpha = \calZ_{\alpha,\vec a}$ defined in \eqref{eq:calZ}, and that it depends on a sequence $(a_j)_{j\in\N}$ with $a_0 = 0$. 

\begin{proposition}\label{prop:Z}
Assume 
\equh\label{eq:a_j}
|a_j|\le Cj^\beta \qmwith \beta\in\Big[0,\frac\alpha 2\Big),
\eque Then,  $(\calZ_\alpha(t))_{t\ge 0}$ is a family of centered Gaussian random variables with
\equh\label{eq:cov Z}
\cov(\calZ_\alpha(s),\calZ_\alpha(t)) =\alpha t^\alpha \sif k1 \frac{F_{\vec a,s,t}\topp 1(k)\Gamma(k-\alpha)}{k!} -  \alpha(s+t)^\alpha \sif k2 \frac{F_{\vec a,s,t}\topp 2(k)\Gamma(k-\alpha)}{k!},
\eque
where $0<s<t$, 
and 
\begin{align*}
F_{\vec a,s,t}\topp1(k) &:= \summ j1k\binom kj \pp{\frac st}^j\pp{1-\frac st}^{k-j}a_ja_k,\\
F_{\vec a,s,t}\topp2(k) &:= \summ j1{k-1}\binom kj \pp{\frac s{s+t}}^j\pp{\frac t{s+t}}^{k-j}a_ja_{k-j}.
\end{align*}
If furthermore 
\equh\label{eq:a_j increment}
|a_{i+j}-a_i|\le Cj^\beta \mfa i,j\in\N  \qmwith \beta\in\Big[0,\frac\alpha 2\Big),
\eque 
 then $(\calZ_\alpha(t))_{t\ge0}$ forms a stochastic process with a version that is $\gamma$-H\"older-continuous for 
  every $\gamma\in(0,\alpha/2)$. 
\end{proposition}
Note that the assumption \eqref{eq:a_j increment} includes the case that $(a_j)_{j\in\N}$ is a sub-additive sequence with polynomial growing rate at most $\beta$. 
 \begin{proof}
We compute directly the covariance, and will see at the end that it is a finite value under the assumption \eqref{eq:a_j}. 
By property of stochastic integrals,
\begin{align}
\cov(\calZ_\alpha(s),\calZ_\alpha(t)) & = \int_0^\infty\esp'\pp{\pp{a_{N'(rs)} - \esp' a_{N'(rs)}}\pp{a_{N'(rt)}-\esp'a_{N'(rt)}}}\alpha r^{-\alpha-1}\d r\nonumber\\
& = \int_0^\infty\esp'\pp{a_{N'(rs)}a_{N'(rt)}}\alpha r^{-\alpha-1}\d r - \int_0^\infty\esp'a_{N'(rs)}\esp'a_{N'(rt)}\alpha  r^{-\alpha-1}\d r.\label{eq:2 integrals}
\end{align}
Here and below, $\esp'$ is the expectation with respect to $\proba'$. The notation $'$ is solely to indicate that they have nothing to do with the probability space where the random measure $\calM_\alpha$ is defined and for convenience from now on, when working with $\esp', \proba'$ and $N'$ without the involvement of $\calM_\alpha$ we may simply write $\esp,\proba$ and $N$ accordingly. 
Recall that $N'(rs)\le N'(rt)$ a.s. since $s<t$. The first integral above equals
\begin{align*}
\int_0^\infty&\sif j1\sif {j'}j a_ja_{j'}\frac{(rs)^j}{j!}e^{-rs} \frac{(r(t-s))^{j'-j}}{(j'-j)!}e^{-r(t-s)}\alpha r^{-\alpha-1}\d r\\
& = \alpha  \sif k1 \summ j1k\binom kj a_ja_ks^j(t-s)^{k-j}\frac 1{k!}\int_0^\infty e^{-rt}r^{k-\alpha-1}\d r\\
& = \alpha t^\alpha \sif k1 \summ j1k\binom kj a_ja_k\pp{\frac st}^j\pp{1-\frac st}^{k-j}\frac{\Gamma(k-\alpha)}{k!} = \alpha t^\alpha \sif k1 F_{\vec a,s,t}\topp1 (k)\frac{\Gamma(k-\alpha)}{k!}.
\end{align*}
Note that in the last step, we used the fact that $|F_{\vec a,s,t}\topp1(k)|\le Ck^{2\beta}$ under the assumption \eqref{eq:a_j}; it then follows that the last series is absolutely convergent (recall that $\Gamma(k-\alpha)/k!\sim k^{-1-\alpha}$ as $k\to\infty$), and hence the interchange of integral and summations in previous steps is also justified by Fubini's theorem.

Similarly, the second integral in \eqref{eq:2 integrals} equals
\begin{align*}
\int_0^\infty&\sif j1\sif {j'}1 a_ja_{j'}\frac{(rs)^j(rt)^{j'}}{j!j'!}e^{-r(s+t)}\alpha r^{-\alpha-1}\d r\\
& = \alpha  \sif k2 \summ j1{k-1}\binom kj a_ja_{k-j}s^jt^{k-j}\frac 1{k!}\int_0^\infty e^{-r(s+t)}r^{k-\alpha-1}\d r\\
& = \alpha (s+t)^\alpha \sif k2 \summ j1{k-1}\binom kj a_ja_{k-j}\pp{\frac s{s+t}}^j\pp{\frac t{s+t}}^{k-j}\frac{\Gamma(k-\alpha)}{k!}\\
& = \alpha (s+t)^\alpha \sif k2 F_{\vec a,s,t}\topp2 (k)\frac{\Gamma(k-\alpha)}{k!},
\end{align*}
and we used the fact that $|F_{\vec a,s,t}\topp2(k)|\le Ck^{2\beta}$ under the assumption \eqref{eq:a_j}. We have thus proved \eqref{eq:cov Z}.

Now we examine the path continuity. By a similar calculation leading to \eqref{eq:2 integrals}, we have
\begin{align*}
\esp \pp{\calZ_\alpha(t)-\calZ_\alpha(s)}^2& = \int_0^\infty \esp\pp{a_{N(rt)}-a_{N(rs)}}^2 \alpha r^{-\alpha-1}\d r\\
& \quad - \int_0^\infty \pp{\esp a_{N(rt)} - \esp a_{N(rs)}}^2\alpha r^{-\alpha-1}\d r\\
& \le \int_0^\infty \esp\pp{a_{N(rt)}-a_{N(rs)}}^2 \alpha r^{-\alpha-1}\d r.
\end{align*}
By \eqref{eq:a_j increment}, it follows that, for all $0<s<t$,
\begin{align*}
\esp \pp{\calZ_\alpha(t)-\calZ_\alpha(s)}^2 &  \le C \int_0^\infty \esp (N(r(t-s)))^{2\beta} \alpha r^{-\alpha-1}\d r\\
& = C\int_0^\infty \sif k1 k^{2\beta}\frac{(r(t-s))^k}{k!}e^{-r(t-s)}\alpha r^{-\alpha-1}\d r\\
& =  C\alpha (t-s)^\alpha\sif k1 \frac{k^{2\beta}\Gamma(k-\alpha)}{\Gamma(k+1)}
 =
 C(t-s)^\alpha.
\end{align*}
Here and below, when $\beta=0$ strictly speaking $0^0$ needs to be defined, and we understand for a Poisson process $N$, $(N(t)-N(s))^0 = 0$ when $N(t)=N(s)$, and $N(t)^0 = 0$ when $N(t) = 0$.
Also, by well-known relation of moments for Gaussian distributions, we have
\[
\esp \pp{\calZ_\alpha(t)-\calZ_\alpha(s)}^{2m} \le C_m \pp{\esp \pp{\calZ_\alpha(t)-\calZ_\alpha(s)}^2}^m \le C(t-s)^{\alpha m}. 
\]
Then for $m>1/\alpha$, Kolmogorov's continuity theorem entails that the process has a continuous version with $\gamma$-H\"older-continuous sample path with $\gamma<(\alpha m-1)/(2m)$. Letting $m\to\infty$, we complete the proof for the second part. 
\end{proof}
\begin{remark}\label{rem:explicit}From stochastic-integral representations, we obtained an explicit formula of the covariance function \eqref{eq:cov Z}. However, most of the time the infinite series in the formulae cannot be simplified and are not directly useful. We mention two noteworthy exceptions, both from \citep[Theorem 2.3]{durieu16infinite}.
\begin{enumerate}[(a)]
\item With $a_j = 1, j\in\N$, $\calZ_\alpha$ is the limit of occupancy process. Then, 
\equh\label{eq:a_j=1}
\cov(\calZ_\alpha(s),\calZ_\alpha(t)) = \Gamma(1-\alpha)\pp{(s+t)^\alpha - \max\{s,t\}^\alpha}.
\eque
\item With $a_j = \inddd{j\rm~is~odd}$, $\calZ_\alpha$ is the limit of odd-occupancy process. Then,
\[
\cov(\calZ_\alpha(s),\calZ_\alpha(t)) = \Gamma(1-\alpha)2^{\alpha-2}\pp{(s+t)^\alpha - |t-s|^\alpha}.
\]
In this case, $(\calZ_\alpha(t))_{t\ge 0}$ is known as the bi-fractional Brownian motion with parameters $(1/2, \alpha)$. See \citep[Remark 2.4]{durieu16infinite} for more discussions. 
\end{enumerate}
We provide a calculation for \eqref{eq:a_j=1} using \eqref{eq:cov Z}. Assume $s<t$. First, 
\begin{align*}
F_{\vec a, s,t}\topp1(k) & = \summ j1k \binom kj \pp{\frac st }^j\pp{1-\frac st}^{k-j} = 1-\pp{1-\frac st}^k,\\
F_{\vec a, s,t}\topp2(k) & = \summ j1{k-1} \binom kj \pp{\frac s{s+t} }^j\pp{\frac t{s+t}}^{k-j} = 1-\pp{\frac t{s+t}}^k - \pp{\frac s{s+t}}^k.
\end{align*}
A key identity in calculations involving limit processes of Karlin model is the following. Let $Q_\alpha$ be an $\alpha$-Sibuya random variable \citep{pitman06combinatorial}. That is, it is an $\N$-valued random variable with 
\[
\proba(Q_\alpha = k) = \frac{\alpha}{\Gamma(1-\alpha)}\frac{\Gamma(k-\alpha)}{\Gamma(k+1)}, k\in\N,
\]
and in particular it satisfies $z^\alpha = 1-\esp 
 (1-z)^{Q_\alpha}$ for all $z\in [0,1)$ (this is essentially Newton's generalized binomial theorem). Then, we have 
\[
\alpha t^\alpha \sif k1\frac{F_{\vec a,s,t}\topp 1(k)\Gamma(k-\alpha)}{\Gamma(k+1)} = \Gamma(1-\alpha)t^\alpha \esp\pp{1-\pp{1-\frac st}^{Q_\alpha}} = \Gamma(1-\alpha) s^\alpha,
\] 
and 
\begin{align*}
\alpha (s+t)^\alpha \sif k2\frac{F_{\vec a,s,t}\topp 2(k)\Gamma(k-\alpha)}{\Gamma(k+1)}& = \Gamma(1-\alpha)(s+t)^\alpha \esp\pp{1-\pp{\frac t{s+t}}^{Q_\alpha} - \pp{\frac s{s+t}}^{Q_\alpha}} \\
&= \Gamma(1-\alpha) \pp{s^\alpha+t^\alpha-(s+t)^\alpha}.
\end{align*}
Combining the above we have derived the stated \eqref{eq:a_j=1}. 

\end{remark}
\section{Limit theorem for the Poissonized model} \label{sec:Karlin}
In order to prove Theorem \ref{thm:2}, we proceed by a standard Poissonization approach, as already carried out in \citep{durieu16infinite,durieu20infinite,chebunin22functional}. This section is devoted to the functional central limit theorem for the Poissonized model of the Karlin model, which we first introduce. 

For each $\ell\in\N$, let $N_\ell = (N_\ell(t))_{t\ge 0}$ be a Poisson process with parameter $p_\ell$, and assume that $(N_\ell)_{\ell\in\N}$ are independent. Set
\[
\wt D_{t,j} := \sif \ell1 \inddd{N_\ell(t) = j}, t\ge 0, j\in\N,
\]
and  define
\begin{equation}\label{eq:wt W_n}
    \widetilde{W}_n(t):=\sum_{j=1}^\infty a_j\pp{ \widetilde{D}_{nt,j}-\E \widetilde{D}_{nt,j} }=\sum_{\ell=1}^\infty\left( a_{N_\ell(nt)}-\E a_{N_\ell(nt)}\right).
\end{equation}
We shall prove the following theorem, which is the Poissonized version of Theorem \ref{thm:2}. 
Recall that we assume $p_j\sim \mathsf C_0j^{-1/\alpha}$ as $j\to\infty$, which we do not repeat in the sequel, and we shall denote $\sigma_n^2:= v(n)\sim \mathsf C_0^{\alpha}n^{\alpha}$. See the discussions around \eqref{eq:p RV}. 
\begin{theorem}\label{thm:3}
Assume $(a_j)_{j\in\N}$ with $a_0 = 0$ satisfies $|a_{i+j}-a_i|\le Cj^\beta$ for all $i,j\in\N$ and some $\beta\in[0,\alpha^2/2)$.     Then,
\[
\pp{ \frac{\widetilde{W}_n(t)}{\sigma_n}}_{t\in[0,1]}\weakto \pp{\calZ_\alpha(t)}_{t\in[0,1]}, \mmas n\rightarrow \infty,
\]
    in $D([0,1])$, where the process $\calZ_\alpha$ is as in Proposition \ref{prop:Z}.
\end{theorem}
To establish Theorem \ref{thm:3}, we first prove the convergence of finite-dimensional distributions in Section \ref{sec:fdd Poisson}. Note that in this step we assume only the weaker condition $|a_j|\le Cj^\beta$ for $\beta\in[0,\alpha/2)$. Next, in Section \ref{sec:tightness} we establish the tightness, and it is here that we impose the stronger condition $|a_{i+j}-a_i|\le Cj^\beta$ with $\beta\in[0,\alpha^2/2)$. 
 \subsection{Finite-dimensional convergence for the Poissonized model}\label{sec:fdd Poisson}
In this section, we prove the following.
\begin{proposition}\label{prop:fdd}
    Suppose that  $(a_j)_{j\in\N}$ satisfies $|a_j|\le Cj^\beta$ for some $\beta\in[0,\alpha/2)$. Then,
\[
        \pp{\frac{\widetilde{W}_n(t)}{\sigma_n}}_{t\in [0,1]}\xrightarrow[]{f.d.d} \pp{\calZ_\alpha(t)}_{t\in[0,1]}, \text{ as } n\rightarrow \infty.
\]
\end{proposition}
We proceed by the following approximation. 
Throughout this subsection without loss of generality assume that $(p_\ell)_{\ell\in\N}$ is non-increasing. For every fixed $\epsilon>0$, write
\begin{equation*}
\widetilde{W}_n(t)=\widetilde{W}^{(\epsilon)}_n(t)+\widetilde{T}^{(\epsilon)}_{n}(t), \quad t\in[0,1], n\in\N,
\end{equation*}
where
\begin{equation*}
    \widetilde{W}^{(\epsilon)}_n(t):= \sum_{\ell=\ell_n+1}^{\infty} \pp{{a}_{N_\ell(nt)} - \esp a_{N_\ell(nt)}}\qmand
    \widetilde{T}^{(\epsilon)}_{n}(t):=\sum_{\ell=1}^{\ell_n} \pp{{a}_{N_\ell(nt)} - \esp a_{N_\ell(nt)}},
\end{equation*}
with
\[
\ell_n\equiv \ell_{n,\epsilon}:=\floor{\epsilon v(n)} = \floor{\epsilon \sigma_n^2}.
\] 
Introduce also
\[
\calZ_\alpha\topp\epsilon(t):    =\int_{(0,\epsilon^{-1/\alpha}]\times \Omega'}\pp{a_{N'(rt)}-\esp' a_{N'(rt)} }\calM_\alpha(\d r,\d \omega'), \quad j\in \N, t\in [0,1].
\]
Clearly, for every $t\ge 0$ fixed, $\calZ_\alpha\topp\epsilon(t)\to \calZ_\alpha(t)$ in probability as $\epsilon\downarrow 0$.

Then, by the standard triangular array approximation argument \citep[Theorem 3.1]{billingsley99convergence}, to prove Proposition \ref{prop:fdd} it suffices to establish Lemmas \ref{CLTepsilon} and \ref{remainderLemma} below. 
\begin{lemma}\label{CLTepsilon} For every $t>0$, under the assumptions of Proposition  \ref{prop:fdd} we have
\[
\pp{\frac{\widetilde{W}_n^{(\epsilon)}(t)}{\sigma_n}}_{t\in[0,1]}\fddto\pp{\mathcal{Z}_\alpha^{(\epsilon)}(t)}_{t\in[0,1]},\quad \text{as } n\rightarrow\infty.
\]
\end{lemma}
\begin{lemma}\label{remainderLemma}
    Under the assumptions of Proposition  \ref{prop:fdd}, for every $t>0$,
    \[
    \lim_{\epsilon\downarrow 0}\limsupn \proba\pp{\abs{\wt T_n\topp\epsilon(t)}>\sigma_n\eta} = 0, \mfa \eta>0. 
       \]
\end{lemma}

Before proving the lemmas we need the following technical estimate which we shall use at several places. Let 
\[
\gamma(a,t) = \int_0^t x^{a-1}e^{-x}\d x, \quad a>0, t>0,
\]
denote the lower incomplete gamma function.

\begin{lemma}\label{TechEst} Let $(b_j)_{j\in\N_0}$ be a sequence such that $b_0 = 0$, $|b_j| \le Cj^{\alpha'}$ for some $\alpha'<\alpha$. Fix~$t>0$. Then, 
    \begin{equation}\label{eq:sum b_n upper}
\limn\frac1{\sigma_n^2}        \sum_{\ell=\ell_n+1}^\infty\E b_{N_\ell(nt)}=  \alpha t^\alpha \sum_{j=1}^\infty \frac{b_j \gamma(j-\alpha,t\epsilon^{-{1}/{\alpha}})}{j!}. 
    \end{equation}
Moreover, there exists a constant $C>0$ such that
       \equh\label{eq:sum b_n lower}
    \summ \ell1{\ell_n}\esp |b_{N_\ell(nt)}| \le C\sigma_n^2t^{\alpha'}\epsilon^{1-\alpha'/\alpha}
    \eque
    for all $n\in\N$. As a consequence, there exists a constant $C>0$ such that for all $t\in(0,1]$, 
\[
    \sif\ell1\esp |b_{N_\ell(nt)}|\le C \sigma_n^2t^{\alpha'}.
\]
\end{lemma}
Note that the limit on the right-hand side of \eqref{eq:sum b_n upper} might be zero.
\begin{proof}
We start by writing
\begin{align*}
	 \sum_{\ell=\ell_n+1}^\infty\E b_{N_\ell(nt)}&=\int_{1/p_{\ell_n}}^\infty \sum_{j=1}^\infty \frac{b_j e^{-{nt}/{x}}}{j!}\pp{\frac{nt}x}^j v(\d x) \\
	 	 &= -\hspace{-.1cm}\int_{1/p_{\ell_n}}^{\infty}\sum_{j=1}^\infty \frac{b_j e^{-nt/x}}{j!} \left(\frac{nt}{x^2}-\frac{j}{x}\right)\left(\frac{nt}{x}\right)^{j} v(x)\d x\hspace{-.05cm}-\hspace{-.05cm} v(p_{\ell_n}\inv)\sum_{j=1}^\infty \frac{b_j e^{-ntp_{\ell_n}} (ntp_{\ell_n})^j  }{j!}\\
	 &= \int_{0}^{ntp_{\ell_n}}\sum_{j=1}^\infty \frac{b_j e^{-z} (j-z)z^{j-1}}{j!} v\left(\frac{nt}{z}\right)\d z- v(p_{\ell_n}\inv)\sum_{j=1}^\infty \frac{b_j e^{-ntp_{\ell_n}} (ntp_{\ell_n})^j  }{j!}\\
	 &=:A_n-B_n.
\end{align*}
Here, the first equality follows from the definition of Riemann--Stieltjes integral and also the assumption that $(p_\ell)_{\ell\in\N}$ is non-increasing, the second  follows from integration by parts, and the third follows from a change of variables.
We claim that 
\begin{equation}\label{epsilonint}
    \lim_{n\rightarrow\infty}\frac{A_n}{\sigma_n^2}= t^\alpha \int_0^{t\epsilon^{-{1}/{\alpha}}} \sum_{j=1}^\infty \frac{b_je^{-z}(j-z)z^{j-1-\alpha}}{j!}\d z,
\end{equation} 
and
\begin{equation}\label{epsilonterm1}
    \lim_{n\rightarrow\infty}\frac{B_n}{\sigma_n^2}= \epsilon \sum_{j=1}^\infty \frac{b_j e^{-t\epsilon^{-{1}/{\alpha}}} (t\epsilon^{-{1}/{\alpha}})^j }{j!}:=R_\epsilon.
\end{equation}

To prove $\eqref{epsilonint}$ note that the function 
\begin{equation*}
    g_n(z):=\ind_{[0,ntp_{\ell_n}]}(z)\sum_{j=1}^\infty \frac{b_j e^{-z} (j-z)z^{j-1}}{j!}\frac{v(nt/z)}{v(n)},
\end{equation*}
converges pointwise to the integrand on the right-hand side of \eqref{epsilonint} multiplied by $t^\alpha$. It remains to justify the domination condition in the dominated convergence theorem. To this end note that since 
\[
ntp_{\ell_n}\sim nt\mathsf C_0\ell_n^{-1/\alpha}\sim nt\mathsf C_0\pp{\epsilon\mathsf C_0^\alpha n^\alpha}^{-1/\alpha}\rightarrow t \epsilon^{-1/\alpha},
\] then for some $\delta'>0$ for $n$ large enough such that $ntp_{\ell _n}<t\epsilon^{-1/\alpha}+\delta'$, and 
an application of Potter's theorem 
(see Theorem 1.5.3 in \cite{bingham87regular}) on $v(nt/z)/v(nt)$ we have for 
$\delta\in (0,1-\alpha)$ and some constant $C>0$ depending on $\delta$,

\begin{equation*}
    |g_n(z)|\leq C\ind_{[0,t\epsilon^{-{1}/{\alpha}}+\delta']}(z)\sum_{j=1}^\infty \frac{|b_j| e^{-z} (j+z)z^{j-1-\alpha}}{j!}\max\ccbb{z^\delta,z^{-\delta}}:= g(z),      
\end{equation*}
for all $n$ large enough, and $g$ is integrable.

To prove \eqref{epsilonterm1}, we write
\[
    \frac{B_n}{\sigma_n^2}=\frac{v(p_{\ell_n}\inv)}{\sigma_n^2}\sum_{j=1}^\infty \frac{b_j (ntp_{\ell_n})^je^{-ntp_{\ell_n}}}{j!},
\]
and observe that $ntp_{\ell_n}\rightarrow t\epsilon^{-1/\alpha}$ and  
\[
v(p_{\ell_n}\inv) \sim \ell_n.
\]
Then, \eqref{epsilonterm1} follows again by the dominated convergence theorem. 

Next, write $\epsilon_0=t\epsilon^{-{1}/{\alpha}}$. Note that the right-hand side of \eqref{epsilonint} can be written as
\begin{align*}
    t^\alpha \sum_{j=1}^\infty \frac{b_j (j\gamma(j-\alpha,\epsilon_0)-\gamma(j+1-\alpha,\epsilon_0))}{j!}&=t^\alpha \sum_{j=1}^\infty \frac{b_j(\alpha\gamma(j-\alpha,\epsilon_0)+\epsilon_0^{j-\alpha} e^{-\epsilon_0})}{j!}\\
    &=t^\alpha \sum_{j=1}^\infty \frac{b_j \alpha\gamma(j-\alpha,t\epsilon^{-{1}/{\alpha}})}{j!}+R_\epsilon,
\end{align*}
where we used $s\gamma(s,x) - \gamma(s+1,x) = x^se^{-x}$ in the first equality. Then combining \eqref{epsilonint} and \eqref{epsilonterm1} we have
\begin{align*}
    \lim_{n\rightarrow\infty}\frac{1}{\sigma_n^2}\sum_{\ell=\ell_n+1}^\infty\E b_{N_\ell(nt)}&=\lim_{n\rightarrow\infty}\left( \frac{A_n}{\sigma_n^2}-\frac{B_n}{\sigma_n^2}\right)
    =t^\alpha \sum_{j=1}^\infty \frac{b_j \alpha\gamma(j-\alpha,t\epsilon^{-1/\alpha})}{j!}+R_\epsilon - R_\epsilon\\
    &=\alpha t^\alpha \sum_{j=1}^\infty \frac{b_j \gamma(j-\alpha,t\epsilon^{-1/\alpha})}{j!}.
\end{align*}
For the second part, we have
\begin{align*}
\summ \ell1{\ell_n}\esp 
|b_{N_\ell(nt)}|&\le C\summ \ell1{\ell_n}\esp N_\ell(nt)^{\alpha'} \le C \summ \ell1{\ell_n} (\esp N_\ell(nt))^{\alpha'}  = C\summ \ell1{\ell_n}(p_\ell nt)^{\alpha'}\\
& \le  C(nt)^{\alpha'} \ell_n^{1-\alpha'/\alpha} \le C\epsilon^{1-\alpha'/\alpha}\sigma_n^2 t^{\alpha'},
\end{align*}
where in the second inequality we used $\esp X^{\alpha'}\le (\esp X)^{\alpha'}$ with $\alpha'<1$ for non-negative random variable $X$. 
\end{proof}
 
\begin{proof}[Proof of Lemma~\ref{CLTepsilon}]
     By the Cram\'er--Wold device it suffices to show that, for any 
      $d\in \N$, distinct $t_1\dots,t_d\in [0,1]$ and any $c_1,\dots,c_d\in\R$, 
\[
        \sum_{i=1}^d c_i\frac{\widetilde{W}_n\topp\epsilon(t_i)}{\sigma_n}\Rightarrow \sum_{i=1}^d c_i  \calZ_\alpha\topp\epsilon(t_i).
\]
Note that
  \[
  \summ i1d c_i \wt W_n\topp \epsilon(t_i) = \sif\ell{\ell_n+1}\summ i1d c_i\pp{a_{N_\ell(nt_i)}-\esp a_{N_\ell(nt_i)}}
  \]
  is a summation of independent random variables, and hence we can apply the central limit theorem for an infinite triangular array of independent random variables satisfying a  Lyapunov-type condition \eqref{lypanovCond} below (\citep[Theorem 8.4.1]{borovkov13probability}). 
  
We first show the convergence of variance. Since
    \begin{align*}
        \Var\left(\sum_{i=1}^d c_i\widetilde{W}^{(\epsilon)}_n(t_i)\right)&=\sum_{i,j=1}^d c_ic_j\Cov \pp{\widetilde{W}^{(\epsilon)}_n(t_i),\widetilde{W}^{(\epsilon)}_n(t_j)},
    \end{align*}
    it suffices to show, for $s<t$, 
    \equh\label{eq:cov =}
\limn\frac1{\sigma_n^2}    \cov\pp{\wt W_n\topp\epsilon(s),\wt W_n\topp\epsilon(t)} = \cov\pp{\calZ_\alpha\topp\epsilon(s),\calZ_\alpha\topp\epsilon(t)}.
\eque
First, we provide the expression of the covariance function in the limit:
\begin{multline*}
\cov\pp{\calZ\topp{\epsilon}_\alpha(s),\calZ\topp{\epsilon}_\alpha(t)}    \\
 = \alpha t^\alpha \sif k1 \frac{F_{\vec a,s,t}\topp 1(k)\gamma(k-\alpha,t\epsilon^{-1/\alpha})}{k!} -  \alpha(s+t)^\alpha \sif k1 \frac{F_{\vec a,s,t}\topp 2(k)\gamma(k-\alpha,(t+s)\epsilon^{-1/\alpha})}{k!}.
    \end{multline*}
The calculation is essentially the same as in Proposition \ref{prop:Z},  replacing $\int_0^\infty$ by $\int_0^{\epsilon^{-1/\alpha}}$. We omit the details.
Next, we compute the limit covariance. Write
\[
\cov\pp{\wt W_n\topp\epsilon(s),\wt W_n\topp\epsilon(t)} = \sif\ell{\ell_n+1}\esp \pp{a_{N_\ell(ns)}a_{N_\ell(nt)}} - \sif\ell{\ell_n+1}\esp a_{N_\ell(ns)}\esp a_{N_\ell(nt)}.
\]
We have
\begin{align*}
\sif\ell{\ell_n+1}\esp \pp{a_{N_\ell(ns)}a_{N_\ell(nt)}} & = \sif\ell{\ell_n+1}\sif j1\sif {j'}0a_ja_{j+j'}\frac{(nsp_\ell)^j(n(t-s)p_\ell)^{j'}}{j!j'!}e^{-ntp_\ell}\\
& = \sif\ell{\ell_n+1}\sif k1\summ j1{k}a_ja_k\frac{(nsp_\ell)^j(n(t-s)p_\ell)^{k-j}}{j!(k-j)!}e^{-ntp_\ell}\\
& = \sif\ell{\ell_n+1}\sif k1 \frac{F_{\vec a,s,t}\topp1(k)(ntp_\ell )^k}{k!}e^{-ntp_\ell} = \sif\ell{\ell_n+1}\esp F_{\vec a,s,t}\topp 1(N_\ell(nt)).
\end{align*}
 Thus, by Lemma \ref{TechEst} with $b_k = F_{\vec a,s,t}\topp 1(k)$,  we have
\[
\limn\frac1{\sigma_n^2}\sif\ell{\ell_n+1}\esp \pp{a_{N_\ell(ns)}a_{N_\ell(nt)}}= \alpha t^\alpha \sif k1 \frac{F_{\vec a,s,t}\topp1(k)\gamma(k-\alpha, t\epsilon^{-1/\alpha})}{k!}.
\]
Similarly
\begin{align*}
\sif\ell{\ell_n+1}\esp a_{N_\ell(ns)}\esp a_{N_\ell(nt)} & = \sif\ell{\ell_n+1}\sif j1\sif{j'}1a_ja_{j'}\frac{(nsp_\ell)^j(ntp_\ell)^{j'}}{j!j'!}e^{-n(s+t)p_\ell}\\
& = \sif\ell{\ell_n+1}\sif k1 F_{\vec a,s,t}\topp 2(k)\frac{(n(s+t)p_\ell)^k}{k!}e^{-n(s+t)p_\ell},
\end{align*}
and hence
\[
\limn
\frac1{\sigma_n^2}\sif\ell{\ell_n+1}\esp a_{N_\ell(ns)}\esp a_{N_\ell(nt)} 
 =\alpha(s+t)^\alpha \sif k1 \frac{F_{\vec a,s,t}\topp 2(k)\gamma(k-\alpha,(t+s)\epsilon^{-1/\alpha})}{k!}.
\]
(Recall that $|F_{\vec a,s,t}\topp i(k)|\le Ck^{2\beta}, i=1,2$ by assumption.) We have thus proved \eqref{eq:cov =}.

To complete the proof of the central limit theorem it remains to verify
\begin{equation}\label{lypanovCond}
  \limn    \frac{1}{\sigma_n^{2+\delta}}\sum_{\ell=\ell_n+1}^\infty \esp\abs{\summ i1d c_i\pp{a_{N_\ell(nt_i)}-\E a_{N_\ell(nt_i)}}}^{2+\delta} =0,
\end{equation}
for some $\delta>0$. We have
\begin{align*}
\frac1{\sigma_n^{2+\delta}}\sum_{\ell=\ell_n+1}^\infty \esp\abs{\summ i1d c_i\pp{a_{N_\ell(nt_i)}-\E a_{N_\ell(nt_i)}}}^{2+\delta}& \le \frac C{\sigma_n^{2+\delta}}\sif\ell{\ell_n+1}\summ i1d \esp\abs{a_{N_\ell(nt_i)}-\E a_{N_\ell(nt_i)}}^{2+\delta}\\
& \le \frac C{\sigma_n^{2+\delta}} \sif\ell{\ell_n+1}\summ i1d \esp |a_{N_\ell(nt_i)}|^{2+\delta}\le \frac {C\sigma_n^2}{\sigma_n^{2+\delta}}\to 0,
\end{align*}
where in the second inequality, we used the fact that, by convexity and Jensen's inequality,  for random variable $X$,   
\equh\label{eq:Jensen}
\E|X-\esp X|^{2+\delta}\leq 2^{1+\delta}\spp{\E|X|^{2+\delta}+|\E X|^{2+\delta}}\leq 2^{2+\delta} \E |X|^{2+\delta},
\eque 
and in the third inequality we used Lemma \ref{TechEst} under the assumption $b_j = |a_j|^{2+\delta} \le C j^{\alpha'}$ for some $\alpha'<\alpha$: for this it suffices to take $\delta>0$ such that $\alpha' = \beta(2+\delta)<\alpha$. 
Therefore, \eqref{lypanovCond} is verified.\end{proof}

\begin{proof}[Proof of Lemma~\ref{remainderLemma}]
We shall prove,         there exist constants $C>0$ and $(d_\epsilon)_{\epsilon>0}$ with $\lim_{\epsilon\downarrow0}d_\epsilon=0$ such that 
  \[
        \Var(\widetilde{T}_{n}^{(\epsilon)}(t))\leq C \sigma_n^2d_\epsilon, \text{ for all } n\in \N.
\]
This then implies the desired result. 
This is then an immediate consequence of \eqref{eq:sum b_n lower}, which yields
\begin{equation*}
\Var\left(\widetilde{T}_{n}^{(\epsilon)}(t)\right)\leq \sum_{\ell=1}^{\ell_n}\E (a_{N_\ell(nt)}^2)\le C \epsilon^{1-{2\beta}/{\alpha}} n^\alpha t^{2\beta}.
\end{equation*}
\end{proof}

\subsection{Tightness for the Poissonized model}\label{sec:tightness}
To prove tightness, we proceed by another approximation of $\wt W_n$. Set 
\[
j_n := \floor{\frac n{\log n}}, n\ge 2,
\]
and $j_1 := 1$. Consider the following truncated process
\[
\what W_n(t):=\sif\ell1 \pp{\what a_{n,\ell}(t) - \esp \what a_{n,\ell}(t)} \qmwith \what a_{n,\ell}(t):=a_{N_\ell(nt)}\inddd{N_\ell(n)\le j_n}, t\in[0,1].
\]
Therefore, in order to show that $\sigma_n\inv(\wt W_n(t))_{t\in[0,1]}$ is tight in $D([0,1])$, we shall establish that both 
\[
\pp{\frac{\what W_n(t)}{\sigma_n}}_{t\in[0,1]} \qmand \pp{\frac{\wt W_n(t) - \what W_n(t)}{\sigma_n}}_{t\in[0,1]}
\]
are tight. These are established in Propositions \ref{prop:tight1} and \ref{prop:tight2} below respectively.

\begin{proposition}\label{prop:tight1} Assume $|a_{i+j}-a_i| \le Cj^\beta$ with $\beta\in[0,\alpha/2)$. 
\begin{enumerate}[(i)]
\item For all 
$p\in[1,\alpha/{2\beta})$,
 we have, for some constant $C$, 
\equh\label{eq:increment moment}
\esp\abs{\what W_n(t)-\what W_n(s)}^{2p} \le C \pp{\sigma_n^2|t-s|^{\alpha} + \sigma_n^{2p}|t-s|^{\alpha p}}, \mfa 0\le s<t\le 1, n\in\N.
\eque
\item For some constant $C>0$, almost surely, 
\equh\label{eq:lem 3.5}
\abs{\what W_n(t)-\what W_n(s)}\le C (N(nt)-N(ns) + n(t-s)), \mfa 0\le s\le t\le 1.
\eque 
\end{enumerate}
As a consequence, if in addition 
\[
\beta\in\bigg[0,\frac{\alpha^2}2\bigg) \qmand p\in\pp{\frac1\alpha,\frac{\alpha}{2\beta}}, 
\]
the process $\sigma_n\inv\what W_n$ is tight in $D([0,1])$. 
\end{proposition}
\begin{proof}
We first prove \eqref{eq:increment moment}. Write
\[
\wb a_{n,\ell}(t) = \what a_{n,\ell}(t) - \esp \what a_{n,\ell}(t). 
\] 
By Rosenthal's inequality (see \cite{merlevede13rosenthal}), we have
\equh\label{eq:Rosenthal}
\esp\abs{\what W_n(t)-\what W_n(s)}^{2p}\le C\pp{\sum_{\ell=1}^\infty \esp \abs{\wb a_{n,\ell}(t) - \wb a_{n,\ell}(s)}^{2p} + \pp{\sum_{\ell=1}^\infty \esp \abs{\wb a_{n,\ell}(t) - \wb a_{n,\ell}(s)}^2}^p}.
\eque
Recall \eqref{eq:Jensen}. 
Observe also that
\[
\abs{\what a_{n,\ell}(t) -  \what a_{n,\ell}(s)}\le  C\pp{N_\ell(nt)-N_\ell(ns)}^\beta\inddd{N_\ell(n)\le j_n} \eqd C N_\ell^\beta(n(t-s))\inddd{N_\ell(n)\le j_n},
\]
where the first step is in the almost sure sense and we used the property $|a_{i+j}-a_i|\le Cj^\beta$, and the second step we used the exchangeability of Poisson random measure. 
Applying \eqref{eq:Jensen} first and combining all the above, we arrive at
\equh\label{eq:p}
\esp\abs{\wb a_{n,\ell}(t) - \wb a_{n,\ell}(s)}^{2p}\le C \esp \abs{\what a_{n,\ell}(t) -  \what a_{n,\ell}(s)}^{2p} \le C \esp \pp{N_\ell^{2\beta p}(n(t-s))\inddd{N_\ell(n)\le j_n}}.
\eque
Taking $p$ such that $2\beta p<\alpha$, and applying  Lemma \ref{lem:moments what a} below, we complete the proof of \eqref{eq:increment moment}.

We next prove \eqref{eq:lem 3.5}. Recall we assume that $|a_{i+j}-a_i|\le Cj^\beta$ for some constant $C$ for all $i,j\in\N$ and hence $|a_{i+j}-a_i|\le Cj$ for all $i,j\in\Z_{\ge0}$. Thus,
\begin{align*}
\abs{\what W_n(t)-\what W_n(s)} 
& \le \sif \ell 1\abs{a_{N_\ell(nt)}-a_{N_\ell(ns)}} + \sif \ell 1\esp\abs{a_{N_\ell(nt)}-a_{N_\ell(ns)}}\\
& \le C (N(nt)-N(ns) + n(t-s)).
\end{align*}
The claimed tightness now follows from Lemma \ref{lem:Olivier} below.
\end{proof}
\begin{lemma} \label{lem:moments what a} 
For every $r\in(0,\alpha/\beta)$, there exists $C>0$ such that
\[
\sif\ell1 \esp \pp{N_\ell(nt)^{\beta r}\inddd{N_\ell(nt)\le n/\log n}}\le  C\sigma_n^2 t^\alpha \quad \mfa n\in\N, t\in[0,1].
        \]
\end{lemma}
The upper bound is to be compared with \eqref{eq:sum b_n upper} and \eqref{eq:sum b_n lower} where the upper bound for the above without the indicator function was obtained.  Without the indicator function, the upper bound in \eqref{eq:sum b_n upper} is of the same order while in \eqref{eq:sum b_n lower} we have obtained an upper bound of 
  order $C\sigma_n^2t^{\beta r}$. The gain here is we replace $t^{\beta r}$ by $t^\alpha$ (recall $r\beta<\alpha$). 
\begin{proof}
We start by writing
\begin{align*}
\sif\ell1 \esp \pp{N_\ell(nt)^{\beta r}\inddd{N_\ell(nt)\le n/\log n}}    &    =
        \int_{0}^\infty \sum_{j=1}^{j_n} \frac{j^{\beta r} e^{-{nt}/{x}}}{j!}\pp{\frac{nt}x}^j v(\d x)\nonumber\\
    &= \int_{0}^{\infty}\sum_{j=1}^{{j_n}} \frac{j^{\beta r} e^{-z} (j-z)z^{j-1}}{j!} v\left(\frac{nt}{z}\right)\d z,
    \end{align*}
   where
in the second step we used integration by parts and the substitution $z=nt/x$. 
Modifying the argument for the asymptotics of $A_n$ in \eqref{epsilonint} in the proof of Lemma \ref{TechEst}, we can show
\[
\int_0^1\sum_{j=1}^{j_n} \frac{j^{\beta r} e^{-z} (j-z)z^{j-1}}{j!} v\left(\frac{nt}{z}\right)\d z \le C \sigma_n^2t^\alpha \quad \mfa n\in\N, t\in[0,1].
\]
(Indeed, the left-hand side above takes the same form of $A_n$ with $b_j$ replaced by $j^{\beta r}$ and the domain of integral $[0,ntp_{\ell_n}]$ replaced by $[0,1]$. Recall that $\limn ntp_{\ell_n} = t\epsilon^{-1/\alpha}$.)
For the integral over $[1,\infty)$, we first mention an identity:  for any sequence $(D_j)_{j\ge 0}$ with $D_0 = 0$ and $j_n\in\N$,
\equh\label{eq:D trick}
\summ j1{j_n}\frac{D_j}{j!}(j-z)z^{j-1}  = \summ j0{j_n}\frac{D_{j+1}-D_j}{j!}z^{j} - \frac{D_{j_n+1}}{j_n!}z^{j_n}.
\eque
Then, it follows that  for $\delta\in(0,\alpha-r\beta)$ and 
 constant $C$ depending on $\delta$,
\begin{align*}
\int_1^\infty\sum_{j=1}^{j_n} \frac{j^{\beta r} e^{-z} (j-z)z^{j-1}}{j!} v\left(\frac{nt}{z}\right)\d z
& \le  v(nt) \int_1^\infty\summ j0{j_n}{\frac{(j+1)^{\beta r}-j^{\beta r}}{j!} z^j}e^{-z}\frac{v(nt/z)}{v(nt)}\d z\\
   & \le Cv(nt)\int_1^\infty\summ j0{j_n}\frac{(j+1)^{\beta r}-j^{\beta r}}{j!} z^{j-\alpha+\delta} e^{-z}\d z,
\end{align*}
where with $D_j = j^{\beta r}$ we applied \eqref{eq:D trick} in the first step (and removed the last negative term) and then applied Potter's bound in the second step  (here we needed the key property that $D_{j+1}-D_j\ge 0$).
It remains to show that the integral above is bounded by a constant that does not 
depend on $n$. Indeed, 
\begin{align*}
\int_1^\infty &\summ j0{j_n}\frac{(j+1)^{\beta r}-j^{\beta r}}{j!} z^{j-\alpha+\delta} e^{-z}\d z \\
& = \int_1^\infty \summ j1{j_n}\frac{j^{\beta r} (j-z)e^{-z}z^{j-1-\alpha+\delta}}{j!}\d z + \frac{(j_n+1)^{\beta r}}{j_n!}\int_1^\infty z^{j_n-\alpha+\delta}e^{-z}\d z\\
& \le \sif j1\frac{j^{\beta r}}{j!}(\alpha-\delta)\Gamma(j-\alpha+\delta) + \frac{(j_n+1)^{\beta r}\Gamma(j_n+1-\alpha+\delta)}{\Gamma(j_n+1)},
\end{align*}
which is finite since $\beta r+\delta<\alpha$. (Note that in the last step above for the first term, we first exchanged the integral and summation, next bounded $\summ j1{j_n}\int_1^\infty$ by $\sif j1\int_0^\infty$, and then applied $j\Gamma(j-\alpha+\delta)-\Gamma(j+1-\alpha+\delta) = (\alpha-\delta)\Gamma(j-\alpha+\delta)$; if we first extend the summation and integral and then exchange the order, then Fubini's theorem does not apply.)
    \end{proof}
\begin{remark}
The discussions around \eqref{eq:Rosenthal} and \eqref{eq:p} are under the assumption $\beta<\alpha^2/2$, which is needed to establish the moment inequality on the increments in \eqref{eq:increment moment}. Another method for establishing tightness is found in \citep{chebunin16functional}, but their approach also requires the moment estimate \eqref{eq:increment moment}.
\end{remark}
\begin{lemma}\label{lem:Olivier}
Let $\vv G_n = (G_n(t))_{t\in[0,1]}, n\in\N$ be a sequence of stochastic processes defined on a common probability space with a standard Poisson process $(N(t))_{t\ge 0}$, satisfying the following two conditions.
\begin{enumerate}[(i)]
\item For some $p>1/\alpha$ we have
\equh\label{eq:3.5}
\esp \abs{G_n(t)-G_n(s)}^{2p}\le C\pp{|t-s|^{\alpha p}\sigma_n^{2p} + |t-s|^\alpha \sigma_n^2}, \quad \mfa s,t\in[0,1], n\in\N.
\eque
\item There exists a constant $C_0>0$ such that almost surely,
\equh\label{eq:3.6}
\abs{G_n(t)-G_n(s)}\le C_0\pp{N(nt)-N(ns)+n(t-s)}, \quad\mfa 0\le s\le t\le1.
\eque
\end{enumerate}
Then,
\[
\lim_{\delta\downarrow 0}\limsupn\proba\pp{\sup_{\substack{s,t\in[0,1]\\
|t-s|\le \delta}}|G_n(t)-G_n(s)|\ge 9\eta\sigma_n}   = 0 \mfa \eta>0.
\]
\end{lemma}
This lemma is an adaption of \citep[Lemma 3.6]{durieu16infinite}. Therein, the lemma assumes a single process $G(t)$ and our lemma here is a generalization in the sense that if the process $G_n(t)$ takes the form $G_n(t)=G(nt)$, then our Lemma \ref{lem:Olivier} becomes \citep[Lemma 3.6]{durieu16infinite}. However, here we shall need $G_n(t) = \what W_n(t)$, and hence we need to slightly modify the tightness argument therein. 
\begin{proof}
Fix $\eta>0$ and $\delta\in(0,1)$. Set $r:=\ceil{1/\delta}$, and $t_i :=i\delta, i=0,\dots,r-1$ and $t_r:=1$. By \citep[Theorem 7.4]{billingsley99convergence}, we have
\equh\label{eq:3.8}
\proba\pp{\sup_{\substack{s,t\in[0,1]\\
|t-s|\le \delta}}|G_n(t)-G_n(s)|\ge 9\eta\sigma_n} \le \summ i1r \proba\pp{\sup_{s\in[t_{i-1},t_i]}|G_n(s)-G_n(t_{i-1})|\ge 3\eta \sigma_n}.
\eque
Fix $i\in\{1,\dots,r\}$ and we analyze the probability on the right-hand side above. Let $\Delta\equiv \Delta_i:=t_i-t_{i-1}$ and set 
\[
x_{k,\ell}\equiv x_{k,\ell}\topp i:=t_{i-1}+\ell\frac{\Delta}{2^k}, \quad k\in\N_0, \ell=0,\dots,2^k.
\]
For each $s\in[t_{i-1},t_i]$ set 
\[
s_k:=\max \ccbb{x_{k,\ell}: \ell=0,\dots,2^k, x_{k,\ell}\le s},\quad k\in\N,
\]
and $s_0:=t_{i-1}$. That is, we have constructed a non-decreasing sequence: $t_{i-1} = s_0\le s_1\le\cdots\le s$.
We also choose 
\equh\label{eq:k_n}
k_n:=\floor{\log_2\pp{2(e-1)\frac{C_0n\Delta}{\eta\sigma_n}}}+1,
\eque
and by triangle inequality we have
\[
\abs{G_n(s)-G_n(t_{i-1})} \le \summ k1{k_n}|G_n(s_k)-G_n(s_{k-1})| + \abs{G_n(s) - G_n(s_{k_n})}.
\]
Now, we derive a uniform upper bound for the difference above {\em for all} $s\in[t_{i-1},t_i]$. For the summation on the right-hand side above, we have
\[
 \summ k1{k_n}|G_n(s_k)-G_n(s_{k-1})|\le \sum_{k=0}^{k_n}\max_{\ell=1,\dots,2^k}|G_n(x_{k,\ell})-G_n(x_{k,\ell-1})|.
\]
This is a compact way of combining two cases: when $s<t_i$ the summation on the right-hand side can be replaced by $\summ k1{k_n}(\cdots)$
(this follows from the observation that $s_k- s_{k-1} = 0$ or $\Delta/2^k$ for $k\ge 1$), and when $s=t_i$ the left-hand side is equal to $|G_n(x_{0,0}) - G_n(x_{0,1})|$ corresponding to $k=0$. We also have
\begin{align}
\abs{G_n(s)-G_n(s_{k_n})}&\le \max_{\ell=0,\dots,2^{k_n}-1}\sup_{s\in[x_{k_n,\ell},x_{k_n,\ell+1}]}C_0\pp{N(ns) - N(nx_{k_n,\ell}) + n(s-x_{k_n,\ell})}\nonumber\\
&\le \max_{\ell=0,\dots,2^{k_n}-1}C_0\pp{N\pp{n(x_{k_n,\ell}+\Delta 2^{-k_n})} - N(nx_{k_n,\ell}) + n\Delta 2^{-k_n}}\nonumber\\
& \le \max_{\ell=0,\dots,2^{k_n}-1}C_0\pp{N\pp{n(x_{k_n,\ell}+\Delta 2^{-k_n})} - N(nx_{k_n,\ell})} + \eta\sigma_n,\nonumber
\end{align}
where in the first inequality we used the assumption \eqref{eq:3.6}, and in the third we recalled the choice of $k_n$ in \eqref{eq:k_n} which implies $C_0 n\Delta 2^{-k_n}<\eta \sigma_n$. We then arrive at
\begin{align}
\limsupn \proba &\pp{\sup_{s\in[t_{i-1},t_i]}|G_n(s)-G_n(t_{i-1})|\ge 3\eta\sigma_n} \nonumber\\
& \le \limsupn \proba\pp{\summ k0{k_n}\max_{\ell=1,\dots,2^k}|G_n(x_{k,\ell})-G_n(x_{k,\ell-1})|>\eta\sigma_n}\label{eq:3.13}\\
& \quad + \limsupn \proba\pp{\max_{\ell=0,\dots,2^{k_n}-1}C_0\pp{N(n(x_{k_n,\ell}+\Delta 2^{-k_n})) - N(nx_{k_n,\ell})}>\eta\sigma_n}.\label{eq:3.14}
\end{align}
(In \citep[Proof of Lemma 3.6]{durieu16infinite}, the analysis of case where $s=t_i$ was missing, and it can be dealt with as shown here.)
The expression of  \eqref{eq:3.14} is zero. Indeed, we have
\begin{align*}
\proba\pp{\max_{\ell=0,\dots,2^{k_n}-1}C_0\pp{N(n(x_{k_n,\ell}+\Delta 2^{-k_n})) - N(nx_{k_n,\ell})}>\eta\sigma_n} & \le 2^{k_n}\proba\pp{N(n\Delta 2^{-k_n})>\frac{\eta \sigma_n}{C_0}}\\
& \le 2^{k_n}e^{n\Delta 2^{-k_n}(e-1)-\eta\sigma_n/C_0} \\
&\le \frac{4C_0(e-1)n\Delta}{\eta\sigma_n}e^{-C'\eta \sigma_n},
\end{align*}
with $C' = (1-1/(2(e-1)))/C_0$. 
In the second inequality above we used Markov's inequality, and in the third we used
\[
2^{k_n}\le \frac{4C_0(e-1)n\Delta}{\eta\sigma_n} \qmand 2^{-k_n}\le \frac{\eta \sigma_n}{2C_0(e-1)n\Delta},
\]
which follows from the choice of $k_n$ in \eqref{eq:k_n}. 

We obtain an upper bound for the expression  \eqref{eq:3.13}. Setting $\eta_k = \eta/((k+1)(k+2)), k\in\N_0$, we have
\begin{align*}
\proba &\pp{\summ k0{k_n}\max_{\ell=1,\dots,2^k}|G_n(x_{k,\ell})-G_n(x_{k,\ell-1})|>\eta\sigma_n}\\
& \le \summ k0{k_n} \proba\pp{\max_{\ell=1,\dots,2^k}|G_n(x_{k,\ell})-G_n(x_{k,\ell-1})|>\eta_k\sigma_n}\\
& \le \summ k0{k_n}\summ \ell1{2^k}\proba\pp{\abs{G_n(x_{k,\ell})-G_n(x_{k,\ell-1})}>\eta_k\sigma_n}.
\end{align*}
By moment estimates in \eqref{eq:3.5}, the last double summation is bounded by 

\begin{align*}
\summ k0{k_n}\summ \ell1{2^k} & \eta_k^{-2p}\frac{\esp|G_n(x_{k,\ell})-G_n(x_{k,\ell-1})|^{2p}}{\sigma_n^{2p}}\\
& \le C\summ k0{k_n}\summ \ell1{2^k}\eta_k^{-2p}\pp{|x_{k,\ell}-x_{k,\ell-1}|^{\alpha p} + \frac{|x_{k,\ell}-x_{k,\ell-1}|^{\alpha}}{\sigma_n^{2(p-1)}}}\\
& \le C \Delta^{\alpha p}\sif k0 \eta_k^{-2p}2^{k(1-\alpha p)} + C \Delta ^\alpha \sigma_n^{-2(p-1)} \summ k0{k_n}\eta_k^{-2p}2^{k(1-\alpha)}.
\end{align*}
In the last expression, the first term is $C\Delta^{\alpha p}$ (since the series is finite), and to bound the second term we use
\[
C\Delta^\alpha\sigma_n^{-2(p-1)}2^{k_n(1-\alpha)}\le 
C\pp{\frac{n}{\sigma_n^{2(p-1)/(1-\alpha)+1}}}^{1-\alpha} \le C\pp{\frac n{\sigma_n^{2/\alpha+1}}}^{1-\alpha}\to 0
\]
as $n\to\infty$, where in the second inequality we used the assumption $\alpha p>1$ (which implies $2(p-1)/(1-\alpha)>2/\alpha$), and 
\[\sum_{k=0}^{k_n} \eta_k^{-2p}\leq C (k_n+2)^{5p},\]
which grows logarithmically. 
It then follows that 
\equh\label{eq:3.13'}
\limsupn \proba \pp{\summ k0{k_n}\max_{\ell=1,\dots,2^k}|G_n(x_{k,\ell})-G_n(x_{k,\ell-1})|>\eta\sigma_n}\le C \Delta ^{\alpha p}.
\eque
Combining \eqref{eq:3.8}, \eqref{eq:3.13}, \eqref{eq:3.14}, \eqref{eq:3.13'}, and since $\Delta\leq \delta$ we have obtained 
\[
\limsupn\proba\pp{\sup_{\substack{s,t\in[0,1]\\
|t-s|\le \delta}}|G_n(t)-G_n(s)|\ge 9\eta\sigma_n} \le C\pp{\floor{\delta\inv}+1}\delta^{\alpha p}. 
\]
Again since $\alpha p>1$, this completes the proof.
\end{proof}

It remains to prove the tightness of $\sigma_n\inv(\wt W_n-\what W_n)$. 
\begin{proposition}With $\beta<\alpha/2$,
\label{prop:tight2}
\equh\label{eq:remainder wt what}
\sup_{t\in[0,1]}\frac{|\wt W_n(t) - \what W_n(t)|}{\sigma_n} \to 0 \quad\mbox{ in probability as $n\to\infty$.}
\eque
\end{proposition}
\begin{proof}
To see this, pick $K>2$ and set 
\[
\Omega_{n,K}:=\ccbb{\sif \ell1\inddd{N_\ell(n)>j_n}\le K\log n \qmand \sup_{\ell\in\N}N_\ell(n)\le Kn}.
\]
Note that, 
\equh\label{eq:Omega_n,K}
\limn\proba(\Omega_{n,K}) = 1.
\eque
Indeed, recalling that $N(n) = \sif\ell 1N_\ell(n)$ is a  Poisson random variable with parameter $n$, we have
\begin{align*}
\proba\pp{\Omega_{n,K}^c} & \le\proba\pp{\sif\ell1 \inddd{N_\ell(n)>j_n}\ge K\log n} + \proba\pp{\sup_{\ell\in\N}N_\ell(n)>Kn}\\
& \le \proba\pp{\sif \ell1 N_\ell(n)>  j_nK \log n} + \proba\pp{\sif \ell1 N_\ell(n)> Kn}\\
& \le  \proba\pp{N(n)>\frac K2n}+\proba\pp{N(n)>Kn} \to 0
\end{align*}
as $n\to\infty$.

Therefore, thanks to \eqref{eq:Omega_n,K} it suffices to prove
\[
\limn\proba\pp{\sup_{t\in[0,1]}\frac{|\wt W_n(t) - \what W_n(t)|}{\sigma_n}\ind_{\Omega_{n,K}} >\epsilon} =   0.
\]Observe that
\begin{align*}
& \sup_{t\in[0,1]}
\abs{{\wt W_n(t) - \what W_n(t)}}\ind_{\Omega_{n,K}}\\
& \quad = \sup_{t\in[0,1]} 
\abs{\summ \ell1{\infty}\pp{a_{N_\ell(nt)}\inddd{N_\ell(n)>j_n} - \esp \pp{a_{N_\ell(nt)}\inddd{N_\ell(n)>j_n}}}}\ind_{\Omega_{n,K}}\\
& \quad \le \sup_{t\in[0,1]}
 \summ \ell1{\infty}\abs{a_{N_\ell(nt)}\inddd{N_\ell(n)>j_n}}\ind_{\Omega_{n,K}}+ \sup_{t\in[0,1]}
  \summ \ell1{\infty}\esp \abs{a_{N_\ell(nt)}\inddd{N_\ell(n)>j_n}}.
\end{align*}
For the first supremum, we have
\begin{align*}
\sup_{t\in[0,1]}
 \summ \ell1{\infty}\abs{a_{N_\ell(nt)}\inddd{N_\ell(nt)>j_n}}\ind_{\Omega_{n,K}} 
 & \le C\summ \ell1{\infty}N_\ell^\beta(n)\inddd{N_\ell(n)>j_n}\ind_{\Omega_{n,K}} \\
 & \le C \sup_{\ell\in\N}N_\ell^\beta(n) \summ \ell 1{\infty}\inddd{N_\ell(n)>j_n}\ind_{\Omega_{n,K}}\le C(Kn)^\beta K\log n.
\end{align*} 
For the second supremum, we have, for $\delta\in(0,1-\alpha)$, 
\begin{align*}
\sup_{t\in[0,1]}\summ \ell1{\infty}\esp \abs{a_{N_\ell(nt)}\inddd{N_\ell(n)>j_n}}&\le C \sup_{t\in[0,1]}\summ \ell1{\infty}\esp \pp{N_\ell^\beta(nt)\inddd{N_\ell(n)>j_n}} \\
&= C\sif \ell1\esp\pp{N_\ell^\beta\inddd{N_\ell(n)>j_n}} \le C\pp{\frac{\log n}n}^{\alpha+\delta-\beta}\summ \ell1{\infty}\esp N_\ell^{\alpha+\delta}(n) \\
&\le C\pp{\frac{\log n}n}^{\alpha+\delta-\beta}\summ \ell1{\infty}(\esp N_\ell(n))^{\alpha+\delta} 
\le Cn^\beta \log ^{\alpha+\delta-\beta} n,
\end{align*} 
where in the second inequality we applied $\esp (N_\ell^\beta\inddd{N_\ell(n)>j_n}) \le \esp (N_\ell^\beta (N_\ell(n)/j_n)^{\alpha+\delta-\beta})$ (notice $\alpha+\delta-\beta>0$).
The stated result \eqref{eq:remainder wt what} now follows.
\end{proof}
\section{De-Poissonization}\label{sec:dePoissonization}
Recall that $W_n(t)$ in \eqref{eq:W_n} and $\wt W_n(t)$ in \eqref{eq:wt W_n} can be re-written as follows
\begin{align*}
W_n(t) &= \sif\ell1 \pp{a_{K_{\floor{nt},\ell}} - \esp a_{K_{\floor{nt},\ell}}},\\
\wt W_n(t) & = \sif \ell1\pp{a_{N_\ell(nt)} - \esp a_{N_\ell(nt)}}.
\end{align*}
The goal is to prove the convergence of $\sigma_n\inv(W_n(t))_{t\in[0,1]}$ and we have just proved the convergence of its Poissonized version $\sigma_n^{-1}\wt W_n$ in Theorem \ref{thm:3}. In this section we complete the proof of Theorem \ref{thm:2} by showing that the difference of the two processes is negligible in an appropriate sense. First, the two processes can be coupled as follows. Let $(\tau_k)_{k\in\N}$ denote the consecutive arrival times of the Poisson process $(N(t))_{t\ge 0}$ with $N(t) := \sif \ell1N_\ell(t)$ and $\tau_0 = 0$. Set
\[
\lambda_n(t) :=\frac{\tau_{\floor{nt}}}n, t\in[0,1].
\]
Then, $\spp{N_\ell(n\lambda_n(t))}_{t\ge 0,\ell\in\N}\eqd \spp{K_{\floor{nt},\ell}}_{t\ge 0,\ell\in\N}$. Without loss of generality, we assume further the equality is almost surely in this section. 
It follows that 
$\sif\ell1 a_{K_{\floor{nt},\ell}} = \sif\ell1a_{N_\ell(n\lambda_n(t))}$ almost surely.
Therefore, with 
\[
\wt q_\ell(t) :=\esp a_{N_\ell(t)}, t\ge 0, \qmand  q_\ell(n):=\esp a_{K_{n,\ell}}, n\in\N,
\]
we have
\begin{align}
\wt W_n(\lambda_n(t)) & = \sif\ell1\pp{a_{N_\ell(n\lambda_n(t))} - \wt q_\ell(n\lambda_n(t))} = \sif\ell1\pp{a_{K_{\floor{nt},\ell}} - \wt q_\ell(n\lambda_n(t))}\nonumber\\
& = W_n(t) + \sif\ell1\pp{q_\ell(\floor{nt}) - \wt q_\ell(n\lambda_n(t))}.\label{eq:Poissonization}
\end{align}
It is well-known that 
\[
\limn\sup_{t\in[0,1]}\abs{\lambda_n(t) - t} = 0, \mbox{ almost surely,}
\] 
and hence by standard de-Poissonization argument, see for example \citep{billingsley99convergence} and also \citep{durieu16infinite}, it follows that
\[
\frac1{\sigma_n}\pp{\wt W_n(\lambda_n(t))}_{t\in[0,1]} \qmand 
\frac1{\sigma_n}\pp{\wt W_n(t)}_{t\in[0,1]}
\]
have the same limit in $D([0,1])$, and we have proved the convergence of the latter to the desired limit in the previous section. 
Then, in view of \eqref{eq:Poissonization}, to complete the proof of Theorem \ref{thm:2} it remains to establish the following. Again, the regular variation assumption \eqref{eq:p RV} on $(p_j)_{j\in\N}$ is imposed throughout without further mention. 
\begin{proposition}\label{prop:approx}
Assume $|a_{i+j}-a_i|\le Cj^\beta$ for all $i,j\in\N$ with $\beta\in[0,\alpha/2)$. We have
\[
\sup_{t\in[0,1]}\frac1{\sigma_n} \abs{\sif\ell1 \pp{\wt q_\ell(n\lambda_n(t)) - q_\ell(\floor{nt})}}\to 0 \mbox{ in probability as $n\to\infty$.}
\]
\end{proposition}
\begin{proof}
We first show 
\equh\label{eq:small t}
\frac1{\sigma_n}\sup_{t\in[0,n^{-(1-\epsilon)}]}\abs{\sif\ell1\pp{\wt q_\ell(n\lambda_n(t)) - q_\ell(\floor{nt})}} \to 0 \mbox{ in probability,} 
\eque
for $\epsilon\in(0,\alpha/2)$. 
We shall show that $\sup_{t\in[0,n^{-(1-\epsilon)}]}\sif\ell1(|\wt q_\ell(nt)| + |q_\ell(nt)| + |\wt q_\ell(n\lambda_n(t))|) = o(\sigma_n)$ 
in probability.
 To see this, we first notice 
\equh\label{eq:1202}
\sup_{t\in[0,n^{-(1-\epsilon)}]}\sif\ell1 \abs{\wt q_\ell(nt)} \le C\sup_{t\in[0,n^{-(1-\epsilon)}]}\sif\ell1 \esp N_\ell(nt) \le C\sif \ell1 \esp N_\ell(n^{\epsilon})\le C \esp N(n^{\epsilon}) = Cn^{\epsilon},
\eque
and 
$\sup_{t\in[0,n^{-(1-\epsilon)}]}\sif \ell1 |q_\ell(nt)| \le Cn^\epsilon$
 can be obtained similarly. It remains to bound $\sup_{t\in[0,n^{-(1-\epsilon)}]}\sif \ell1|\wt q_\ell(n\lambda_n(t))|$. One can first show that 
 \[
 \limn\proba\pp{\sup_{t\in[0,n^{-(1-\epsilon)}]}n\lambda_n(t)>2n^\epsilon} = 0.
 \]
This is because the supremum 
of interest above
 is achieved at $t=n^{-(1-\epsilon)}$ and equals $n\lambda_n(n^{-(1-\epsilon)}) = \tau_{\sfloor{n^\epsilon}}$, 
 which is a Gamma random variable with parameter $\sfloor{n^{\epsilon}}$.
  Then, to show $\sup_{t\in[0,n^{-(1-\epsilon)}]} \sif \ell1|\wt q_\ell(n\lambda_n(t))|/\sigma_n\to 0$ in probability it suffices to show 
  \[
\sup_{t\in[0,2n^{\epsilon}]}\sif \ell1|\wt q_\ell(t)| = o(\sigma_n),
  \] where the last step follows from \eqref{eq:1202}. We have thus established \eqref{eq:small t}.

It remains to show that for all $\epsilon>0$, 
\equh\label{eq:t large 0}
\limn\sup_{t\in [n^{-(1-\epsilon)},1]} \frac1{\sigma_n} \abs{\sif\ell1 \pp{\wt q_\ell(n\lambda_n(t)) - q_\ell(\floor{nt})}}= 0.
\eque
To better illustrate the analysis, we first deal with $t=1$ (instead of considering the supremum over $t\in[n^{-(1-\epsilon)},1]$). 
A crucial step is to apply a local central limit theorem controlling $|\wt q_\ell(n)-q_\ell(n)| = |\esp a_{N(np_\ell)} -\esp a_{B_{n,p_\ell}}|$ and also $|\wt q_\ell(n+x)-\wt q_\ell(n)| = |\esp a_{N((n+x)p_\ell)} - \esp a_{N(np_\ell)}|$, where $N(\lambda)$ is a Poisson random variable with parameter $\lambda>0$ and $B_{n,p}$ is a binomial random variable with parameters $(n,p)$. This step is technical and is explained in Section \ref{sec:local CLT}.
\medskip

\noindent {\em Proof for the case $t=1$.}
Write $\lambda_n \equiv \lambda_n(1) = \tau_n/n$. Set
\[
A_n:=\ccbb{\abs{n\lambda_n - n}\le n^{1/2}\log n}.
\]
As a consequence of the central limit theorem we have $\limn\proba(A_n^c) = 0$, and it suffices to show
\equh\label{eq:t=1}
\limn \frac1{\sigma_n} \abs{\sif\ell1 \pp{\wt q_\ell(n\lambda_n) - q_\ell(n)}}\ind_{A_n}= 0, \quad \mbox{ almost surely.}
\eque
We have
\[
\sif \ell1 \pp{\wt q_\ell(n\lambda_n) - q_\ell(n)}\ind_{A_n} =\sif\ell1\pp{\wt q_\ell(n\lambda_n) - \wt q_\ell(n)}\ind_{A_n} + \sif\ell1\pp{\wt q_\ell(n) - q_\ell(n)}\ind_{A_n}=: I_{n,1}+I_{n,2}.
\]

We first deal with $I_{n,1}$. 
Set
\[
d_n := n^{1/2}\log n
\qmand m_n:=\sfloor{n^{\alpha'}}, n\in\N,
\]
for some $\alpha'>\alpha/2$ (and satisfying another condition \eqref{eq:alpha' upper0} below). Notice
\[
\abs{\wt q_\ell(n\lambda_n) - \wt q_\ell(n)}\ind_{A_n} \le \sup_{|x|\le d_n}\abs{\esp a_{N_\ell(n+x)}-\esp a_{N_\ell(n)}} \mbox{ almost surely.}
\]
Write $J_{m}:=\{j\in\N: |j-m|\le d_m\}$.
We have
\begin{align*}
\sup_{|x|\le d_n} & \abs{\esp\pp{a_{N((n+x)p)}\inddd{N((n+x)p)\in J_{np}}} - \esp\pp{a_{N(np)}\inddd{N(np)\in J_{np}}}}\\
 & \le \sum_{k\in J_{np}}|a_k|\sup_{|x|\le d_n} \abs{\proba(N((n+x)p) = k) - \proba(N(np) = k)}\\
 & \le C (np)^\beta \log n\pp{O\pp{\sqrt p\log^2n}+O\pp{\frac{\log^3n}{\sqrt{np}}}},
\end{align*}
where in the second inequality we applied a local central limit theorem in Lemma \ref{lem:local CLT}. 
We also have
\[
 \sup_{|x|\le d_n}\esp\abs{a_{N((n+x)p)}\inddd{N((n+x)p)\notin J_{np}}} + \esp\abs{a_{N(np)}\inddd{N(np)\notin J_{np}}} \to 0
\]
faster than any polynomial rate (uniformly for all $p>n^\epsilon$). Indeed, $|a_j|\le Cj^\beta$ and $\proba(N(np)\notin J_{np})\to 0$ decays at rate $e^{-C\log ^2n}$
 (by Cram\'er--Chernoff bounds \cite{boucheron13concentration}), and one can show for $n$ large enough,  
\[
\proba(N((n+x)p)\notin J_{np}) \le \proba\pp{|N((n+x)p) - (n+x)p| >D\sqrt{(n+x)p}\log((n+x)p)},
\]
for some $0<D<1+\sqrt{p_1}$, which again tends to zero at rate $e^{-C\log^2n}$. 
So, we have shown
\equh\label{eq:I_n11}
\sup_{|x|\le d_n}\summ \ell1{m_n}\abs{\esp a_{N_\ell(n+x)} - \esp a_{N_\ell(n)}}  \le Cn^\beta \summ \ell1{m_n}p_\ell^{1/2+\beta}\log^3n + C\summ \ell1{m_n}\frac{\log^4 n}{(np_\ell)^{1/2-\beta}} + o(n^{\alpha/2}).
\eque
Note that when $1/2+\beta\ge \alpha$, the first summation is of order $O(n^\beta\log ^4n) = o(n^{\alpha/2})$. When, $1/2+\beta<\alpha$, the first summation is of order $n^{\alpha'(1-(1/2+\beta)/\alpha)+\beta}$, which is of order $o(n^{\alpha/2})$ if 
\[
\alpha'<\frac\alpha 2\frac{2\alpha-4\beta}{2\alpha-1-2\beta}. 
\]
The second summation on the right-hand side above is of order
$n^{\alpha'(1+(1/2-\beta)/\alpha) - (1/2-\beta)}\log ^4n$,
which is of order $o(n^{\alpha/2})$ when 
\equh\label{eq:alpha' upper0}
\alpha'<\frac\alpha2\frac{2+2\alpha-4\beta}{1+2\alpha-2\beta}.
\eque
One readily checks that this condition is the most restrictive constraint on $\alpha'$. 
In summary, for all $\beta\in[0,\alpha/2)$, under \eqref{eq:alpha' upper0} the upper bound in \eqref{eq:I_n11} is of order $o(n^{\alpha/2})$ almost surely.

 Next, notice that using $|a_{i+j}-a_i|\le Cj^\beta$, 
\begin{align*}
\sif\ell{m_n} \abs{\wt q_\ell(n\lambda_n) - \wt q_\ell(n)}\ind_{A_n}  &\le \sup_{|u-n|\le d_n} 
\sif\ell{m_n} \esp \abs{a_{N_\ell(u)} - a_{N_\ell(n)}}\\
& \le C\sup_{|u-n|\le d_n} 
\sif\ell{m_n} \esp\pp{|N_\ell(u)-N_\ell(n)|} \le C\sif\ell{m_n}\esp N(d_np_\ell).
\end{align*}
Assume $\alpha'>\alpha/2$. Then, $\limn\sup_{\ell\ge m_n}d_np_\ell = 0$, and therefore
\begin{align}
\sif\ell{m_n} \abs{\wt q_\ell(n\lambda_n) - \wt q_\ell(n)}\ind_{A_n}& \le C\sif\ell{m_n}\esp N(d_np_\ell) = Cn^{1/2}\log n\sif\ell{m_n}p_\ell\nonumber\\
&\le Cn^{1/2+\alpha'(1-1/\alpha)}\log n = o(n^{\alpha/2}). \label{eq:I_n12}
\end{align}
We have thus proved, by choosing $\alpha'$ satisfying \eqref{eq:alpha' upper0} and $\alpha'>\alpha/2$,
$
|I_{n,1}| = o(n^{\alpha/2})$
almost surely.

Next, we deal with $I_{n,2}$. We shall again decompose the summation into $\summ\ell1{m_n}$ and $\summ \ell{m_n+1}\infty$ with $m_n = \sfloor{n^{\alpha'}}$. However, this time the choice of $\alpha'$ is not necessarily the same as when dealing with $I_{n,1}$ before (but instead satisfying \eqref{eq:alpha' upper} and \eqref{eq:alpha' lower} below).

We first consider $\ell\le m_n$. 
Recall that $K_{n,\ell}$ is a binomial random variable with parameters $n,p_\ell$ (denoted also by $B_{n,p_\ell}$). 
This time, since $\ell\leq m_n$ implies $p_\ell\in [n^{-\alpha'/\alpha},1/2]$ we can apply Lemma \ref{lem:local CLT}. We have
\begin{align*}
& \abs{\esp\pp{a_{N(np)}\inddd{N(np)\in J_{np}}} - \esp\pp{a_{B_{n,p}}\inddd{B_{n,p}\in J_{np}}}}\\
 & \le \sum_{k\in J_{np}}|a_k|\abs{\proba(N(np) = k) - \proba(B_{n,p} = k)}\\
 & \le C (np)^\beta \log n\pp{O\pp{p\log^2 n}+O\pp{\frac{\log^4 n}{np}}}.
\end{align*}
We also have
$\esp\abs{a_{N(np)}\inddd{N(np)\notin J_{np}}} + \esp\abs{a_{B_{n,p}}\inddd{B_{n,p}\notin J_{np}}} \to 0$
faster than any polynomial rate (we have discussed the first convergence before, and the second follows from Bernstein's inequality). Thus,
\begin{align}
\summ\ell1{m_n}\abs{\esp a_{N_\ell(n)} - \esp a_{K_{n,\ell}}}&\le Cn^\beta\log^3n\summ\ell1{m_n}p_\ell^{1+\beta} + Cn^{\beta-1}\log^4n\summ \ell1{m_n}p_\ell^{\beta-1}+ o(n^{\alpha/2})\nonumber\\
& \le C n^\beta\log^3n+Cn^{\beta-1+\alpha'(1+(1-\beta)/\alpha)}\log^4n + o(n^{\alpha/2}).\label{eq:I_n21}
\end{align}
For the upper bound to be of order $o(n^{\alpha/2})$, we need to impose 
\equh\label{eq:alpha' upper}
\alpha'<\frac\alpha2\frac{\alpha+2-2\beta}{\alpha+1-\beta}.
\eque

For $\ell \ge m_n$, we use a standard Poisson-Binomial coupling and consequently have
\[\E |N(np)-B_{n,p}|\leq np^2,
\]
for all $p\in (0,1)$ and $n\geq 1$. Applying this inequality, and using $|a_k-a_\ell|\le C|k-\ell|^\beta\le C|k-\ell|$ for $0\le \beta<1$ (recall we understand $0^0 = 0$)  we have
\begin{align}
	\sum_{\ell\geq m_n}|\E a_{N(np_\ell)}-\E a_{B_{n,p_\ell}}|&\le \sum_{\ell\ge m_n}\esp\abs{a_{N(np_\ell)}-a_{B_{n,p_\ell}}}\leq \sum_{\ell\geq m_n}C \E|N(np_\ell)-B_{n,p_\ell}| \nonumber\\
	&\leq \sum_{\ell\geq m_n}Cn p_\ell^2 \leq Cn m_n^{1-2/\alpha}\label{eq:I_n22},
\end{align}
and for the upper bound to be of order $o(n^{\alpha/2})$ it is equivalent to impose
\begin{equation}\label{eq:alpha' lower}
	\alpha'>\alpha/2.
\end{equation}
 It remains to check that there exists $\alpha'$ satisfying both \eqref{eq:alpha' upper} and \eqref{eq:alpha' lower}. This is equivalent to the condition that $\beta<1$. But we have already imposed $\beta< \alpha/2$, which is a stronger condition. That is, we have proved that, by choosing $\alpha'$ satisfying \eqref{eq:alpha' upper} and \eqref{eq:alpha' lower}, $|I_{n,2}| =o(n^{\alpha/2})$.
We have completed the proof of \eqref{eq:t=1}.

  \medskip
  
  \noindent{\em Proof of \eqref{eq:t large 0}.}
Now we extend the above upper bounds to uniform controls in both $n$ and $t$. First, one needs to modify accordingly the definition of $A_n$. 
Recall that $\lambda_n(t) = \tau_{\floor{nt}}/n$ and is piecewise constant.  So we have
\[
\sup_{t\in[0,1]}\abs{n\lambda_n(t) - \floor{nt}}= \max_{j=1,\dots,n}\abs{\tau_j-j}.
\]
This time, we fix $\epsilon>0$ and set
\equh\label{eq:A_ne}
A_{n,\epsilon} :=\bigcap_{j=\floor{n^\epsilon}}^n\ccbb{|\tau_j - j|\le \sqrt j\log j}.
\eque
Recall also that $\tau_j$ is the partial sum of i.i.d.~standard exponential random variables. 
The concentration inequality for exponential random variables \citep{talagrand91new,bobkov97poincare} says that $\proba(
\tau_n-n>t) \le e^{-C\min\{t,t^2/n\}}$, and we also have $\proba(\tau_n-n<-t)\le e^{-t^2/(2(n+t/3))}$ by Bernstein's inequality. Thus, we have
\[
\proba(A_{n,\epsilon}^c) \le \summ k{\sfloor{n^\epsilon}}n e^{-C\log^2 k} \to 0
\]
as $n\to\infty$. Thus, in order to prove \eqref{eq:t large 0} it suffices to prove
\equh\label{eq:large t}
\limn\sup_{t\in [n^{-(1-\epsilon)},1]} \frac1{\sigma_n} \abs{\sif\ell1 \pp{\wt q_\ell(n\lambda_n(t)) - q_\ell(\floor{nt})}}\ind_{A_{n,\epsilon}}= 0, \mbox{ almost surely.}
\eque

The analysis now becomes essentially the same as in the proof of the case $t=1$, except that now parameters $m_n$ and $d_n$ are chosen to depend on $t$. Our upper bounds below are uniform in $n\in\N, t\ge n^{-(1-\epsilon)}$. Set 
\[
m_n(t):=\floor{(nt)^{\alpha'}} \qmand d_n(t):=\floor{nt}^{1/2}\log \floor{nt},
\]
and write
\begin{multline*}
\sif \ell1 \pp{\wt q_\ell(n\lambda_n(t)) - q_\ell(\floor{nt})}\ind_{A_{n,\epsilon}} \\
= \sif\ell1\pp{\wt q_\ell(n\lambda_n(t)) - \wt q_\ell(\floor{nt})}\ind_{A_{n,\epsilon}} + \sif\ell1\pp{\wt q_\ell(\floor{nt}) - q_\ell(\floor{nt})} = :   I_{n,1}(t)+I_{n,2}(t).
\end{multline*}
In place of the analysis of $I_{n,1}$ earlier, we now have
\[
\abs{I_{n,1}(t)}  \le \sup_{|u-\floor{nt}|\le d_n(t)}\summ \ell1{m_n(t)}\abs{\wt q_\ell(u) - \wt q_\ell(\floor{nt})} +\sup_{|u-\floor{nt}|\le d_n(t)} \summ\ell{m_n(t)+1}\infty\abs{\wt q_\ell(u) - \wt q_\ell(\floor{nt})}.
\]
This time we obtain similar bounds as in \eqref{eq:I_n11} and \eqref{eq:I_n12} (with $n$ replaced by $\floor{nt}$), which is almost surely of order $o(n^{\alpha/2})$ under $\alpha'>\alpha/2$ and \eqref{eq:alpha' upper0}.
In place of the analysis of $I_{n,2}$ earlier, we now have
\[
\abs{I_{n,2}(t)}  \le \summ\ell1{m_n(t)}\abs{\wt q_\ell(\floor{nt}) - q_\ell(\floor{nt})} + \summ\ell{m_n(t)+1}\infty \abs{\wt q_\ell(\floor{nt}) - q_\ell(\floor{nt})},
\]
and we can obtain similar bounds as in \eqref{eq:I_n21} and \eqref{eq:I_n22}, which is of order $o((nt)^{\alpha/2})$ under \eqref{eq:alpha' upper} and \eqref{eq:alpha' lower}. 
Combining the two upper bounds we have proved \eqref{eq:large t}. 
\end{proof}

\begin{remark}
Proposition \ref{prop:approx} is much more involved than the corresponding step in earlier examples. In \citep{durieu16infinite}, this step is relatively straightforward (especially for the estimate of $I_{n,2}$) as with $a_j = 1$ or $a_j = \inddd{j\rm~is~odd}$, the function $\wt q_\ell$ and $q_\ell$ have explicit formulae of which one can derive a sharp control on the difference above quickly. See Lemmas 4.2 and 4.4 therein. We also notice in passing that in \citep[Step 4 in the proof of Theorem 3]{chebunin16functional}, the statement $\proba(Y_{n,k}^*(t) = Z_{n,k}^*\mid \Pi(n\tau) = \floor{nt}) = 1$ is incorrect (it is correct if $Y_{n,k}^*$ and $Z_{n,k}^*(\tau)$ both are replaced by the same statistic without centering). This small overlook is easy to fix therein as they considered the simple case $a_{j_0} = 1$ for some prescribed $j_0$ and $a_j = 0$ otherwise. It is also worth pointing out that they did not apply the chaining argument for the tightness of the Poissonized process, while they still needed certain moment estimates on the increments as in our \eqref{eq:increment moment}. 
\end{remark}

\begin{remark}
A natural candidate to replace $A_n$ (instead of $A_{n,\epsilon}$ used in \eqref{eq:A_ne}) is 
\equh\label{eq:A_n}
A_n:=\ccbb{\sup_{t\in[0,1]}n^{1/2}\abs{\lambda_n(t)-t}\le \log n}, n\in\N,
\eque
for which we still have $\limn\proba(A_n^c) = 0$, given $n^{1/2}(\lambda_n(t)-t)_{t\in[0,1]}\weakto (\BB_t)_{t\in[0,1]}$ is $D([0,1])$; consequently the sequence $(\sup_{t\in[0,1]}n^{1/2}|\lambda_n(t)-t|)_{n\in\N}$ is tight.
Notice that restricted to the event $A_n$ defined in \eqref{eq:A_n} we have $|n\lambda_n(t)-nt|\le n^{1/2}\log n$. This control could eventually lead to a uniform control in \eqref{eq:large t} but only for the supremum restricted to $t\in[n^{-(1/2-\epsilon)},1]$. In particular, for $t = O(n^{-1/2})$, the upper bound control is not sharp enough. 
\end{remark}

\subsection{Local central limit theorems for Poisson and binomial distributions}\label{sec:local CLT}
The following lemma provides a refined control on the local central limit theorem for Poisson and binomial distributions. For example, local central limit theorem says that 
$\proba(N(n) = k) = (2\pi n)^{-1/2}\exp\pp{-(k-n)^2/(2n)}+o(n^{-1/2})$ and the error term $o(n^{-1/2})$ is uniformly in $k\in\N$. This error control is not sharp enough for our purposes and instead we shall need the following.
We write $d_n:=n^{1/2}\log n $ and $J_{m}:=[m-m^{1/2}\log m,m+m^{1/2}\log m]\cap\N$.
\begin{lemma}\label{lem:local CLT}
Fix $\epsilon\in(0,1)$. Then, there exists $C>0$ such that
\equh\label{eq:local difference}
\max_{k\in J_{np}}\abs{\proba(N(np) = k) - \proba(B_{n,p} = k)} \le \frac C{\sqrt {np}}\pp{p\log^2n+\frac{\log^4n}{np}},
\eque
and
\equh\label{eq:local difference1}
\sup_{|x|\le d_n}\max_{k\in J_{np}}\abs{\proba(N((n+x)p) = k) - \proba(N(np) = k)} \le \frac C{\sqrt {np}}\pp{\sqrt p\log^2n+\frac{\log^3n}{\sqrt{np}}},
\eque 
for all $n\in\N, p\in[ n^{-\epsilon},1/2]$. 
\end{lemma}
Note that the upper constraint $p\le 1/2$ can be relaxed to $p\le b<1$ for any $b$ fixed. The requirement $p\ge n^{-\epsilon}$ with $\epsilon<1$ is such that $\sup_{p\in[n^{-\epsilon},1/2]}np\to\infty$, which will be used at multiple places. 
\begin{proof} 
We first prove \eqref{eq:local difference}. Recall Stirling's formula
\[
n!  = \pp{\frac ne}^n \sqrt{2\pi n}\pp{1+O(n\inv)}.
\]
We start with 
\[
\proba(N(np) = k)  = \frac{(np)^k}{k!}e^{-np} = \frac1{\sqrt{2\pi k}}\pp{1-\frac{k-np}k}^k e^{k-np}\pp{1+O(k\inv)}.
\]
Notice that 
\begin{align*}
\pp{1-\frac{k-np}k}^k & = \exp\pp{k\log\pp{1-\frac{k-np}k}}\\
& = \exp\pp{-(k-np) - \frac{(k-np)^2}{2k} -\frac{(k-np)^3}{3k^2} + O\pp{\frac{(k-np)^4}{k^3}}}.
\end{align*}
The big O term on the right-hand side above is understood as follows: $O((k-np)^4/k^3)$ stands for a term depending on $n,p$ and $k$, say $r_{n,p,k}$, such that for some constant $C>0$, 
\equh\label{eq:O}
 |r_{n,p,k}|\le \frac{C(k-np)^4}{k^3},
\quad \mfa 
n\in\N, p\in[n^{-\epsilon},1/2], k\in J_{np}.
\eque
Similar interpretations apply below. Therefore, we have shown that
\begin{align}
\proba(N(np) = k) 
& = \frac1{\sqrt{2\pi k}}\exp\pp{-\frac{(k-np)^2}{2k} - \frac{(k-np)^3}{3k^2}}\pp{1+ O(k\inv) + O\pp{\frac{(k-np)^4}{k^3}}}\nonumber\\
\label{eq:Poisson local}
&=
\frac1{\sqrt{2\pi k}}\exp\pp{-\frac{(k-np)^2}{2k} - \frac{(k-np)^3}{3k^2}}\pp{1+ O\pp{\frac{\log^4(np)}{np}}}.
\end{align}

Next, we have
\begin{align*}
\proba(B_{n,p} = k) & = \frac{n!}{k!(n-k)!}p^k(1-p)^{n-k}\\
& = \pp{\frac{np}k}^k\pp{\frac {n(1-p)}{n-k}}^{n-k}\sqrt{\frac n{k(n-k)}}\frac1{\sqrt{2\pi}}\pp{1+O(k\inv)}.
\end{align*}
This time, 
$\sqrt{n/(n-k)} = 1+O\pp{k/(n-k)} = 1+O\pp{p}$,
and
\begin{align*}
k\log{\frac{np}k}&+(n-k)\log\frac{n(1-p)}{n-k}  \\
&= k\log\pp{1-\frac{k-np}k} + (n-k)\log\pp{1+\frac{k-np}{n-k}}\\
& = -\frac{(k-np)^2}{2k} - \frac{(k-np)^3}{3k^2} + O\pp{\frac{(k-np)^4}{k^3}}  - \frac{(k-np)^2}{2(n-k)}+ O\pp{\frac{|k-np|^3}{3(n-k)^2}}.
\end{align*}
Notice also that 
\[
\exp\pp{- \frac{(k-np)^2}{2(n-k)}} = 1+O\pp{\frac{(k-np)^2}{2(n-k)}} = 1+O\pp{p\log^2(np)}.
\]
We therefore have
\begin{multline}\label{eq:binomial local}
\proba(B_{n,p} = k) \\
 = \frac1{\sqrt{2\pi k}}\exp\pp{-\frac{(k-np)^2}{2k} - \frac{(k-np)^3}{3k^2}}
 \pp{1+O\pp{p\log^2(np)}+O\pp{\frac{\log^4(np)}{np}}}.
\end{multline}
Combining \eqref{eq:Poisson local} and \eqref{eq:binomial local}, we have proved the stated result \eqref{eq:local difference}.

Next, we prove \eqref{eq:local difference1}. This time we have
\begin{align*}
\proba(N((n+x)p) = k)&  =  \frac1{\sqrt{2\pi k}}\pp{1-\frac{k-(n+x)p}k}^k e^{k-(n+x)p}\pp{1+O(k\inv)}\\
 & = \frac1{\sqrt{2\pi k}}\exp\pp{ - \frac{(k-(n+x)p)^2}{2k} +O\pp{\frac{(k-(n+x)p)^3}{3k^2}}}.
\end{align*}
The big O term is understood similarly as in \eqref{eq:O} but in addition for all $x\in\R$ such that $|x|\le d_n$. A key observation is that for $p$ small, a typical deviation caused by $x$, $xp$, is $O(n^{1/2}p\log n)$, which is much smaller than the typical deviation $k-np$, which is  $O((np)^{1/2}\log(np))$.  Now,
\begin{align*}
 \frac{(k-(n+x)p)^2}{2k} & = \frac{(k-np)^2}{2k} + \frac{(xp)^2}{2k} - \frac{(k-np)xp}{k} \\
& = \frac{(k-np)^2}{2k} + O\pp{p\log^2 n} + O\pp{\sqrt p\log^2 n},
\end{align*}
and, 
\[
\frac{(k-(n+x)p)^3}{k^2} = O\pp{\frac{(\sqrt{np}\log n)^3}{(np)^2}} = O\pp{\frac{\log^3n}{\sqrt{np}}}.
\]
The stated result now follows.
\end{proof}

\bibliographystyle{apalike}
\bibliography{../../include/references,../../include/references18}

\def\cprime{$'$} \def\polhk#1{\setbox0=\hbox{#1}{\ooalign{\hidewidth
  \lower1.5ex\hbox{`}\hidewidth\crcr\unhbox0}}}
  \def\polhk#1{\setbox0=\hbox{#1}{\ooalign{\hidewidth
  \lower1.5ex\hbox{`}\hidewidth\crcr\unhbox0}}}
\begin{thebibliography}{}

\bibitem[Arratia et~al., 2003]{arratia03logarithmic}
Arratia, R., Barbour, A.~D., and Tavar\'e, S. (2003).
\newblock {\em Logarithmic combinatorial structures: a probabilistic approach}.
\newblock EMS Monographs in Mathematics. European Mathematical Society (EMS),
  Z\"urich.

\bibitem[Bahadur, 1960]{bahadur60number}
Bahadur, R.~R. (1960).
\newblock On the number of distinct values in a large sample from an infinite
  discrete distribution.
\newblock {\em Proc. Nat. Inst. Sci. India Part A}, 26(supplement II):67--75.

\bibitem[Bahier, 2019]{bahier19characteristic}
Bahier, V. (2019).
\newblock Characteristic polynomials of modified permutation matrices at
  microscopic scale.
\newblock {\em Stochastic Process. Appl.}, 129(11):4335--4365.

\bibitem[Bahier and Najnudel, 2022]{bahier22smooth}
Bahier, V. and Najnudel, J. (2022).
\newblock On smooth mesoscopic linear statistics of the eigenvalues of random
  permutation matrices.
\newblock {\em J. Theoret. Probab.}, 35(3):1640--1661.

\bibitem[Ben~Arous and Dang, 2015]{benarous15fluctuations}
Ben~Arous, G. and Dang, K. (2015).
\newblock On fluctuations of eigenvalues of random permutation matrices.
\newblock {\em Ann. Inst. Henri Poincar\'{e} Probab. Stat.}, 51(2):620--647.

\bibitem[Billingsley, 1999]{billingsley99convergence}
Billingsley, P. (1999).
\newblock {\em Convergence of probability measures}.
\newblock Wiley Series in Probability and Statistics: Probability and
  Statistics. John Wiley \& Sons Inc., New York, second edition.
\newblock A Wiley-Interscience Publication.

\bibitem[Bingham et~al., 1987]{bingham87regular}
Bingham, N.~H., Goldie, C.~M., and Teugels, J.~L. (1987).
\newblock {\em Regular variation}, volume~27 of {\em Encyclopedia of
  Mathematics and its Applications}.
\newblock Cambridge University Press, Cambridge.

\bibitem[Bobkov and Ledoux, 1997]{bobkov97poincare}
Bobkov, S. and Ledoux, M. (1997).
\newblock Poincar\'e's inequalities and {T}alagrand's concentration phenomenon
  for the exponential distribution.
\newblock {\em Probab. Theory Related Fields}, 107(3):383--400.

\bibitem[Borovkov, 2013]{borovkov13probability}
Borovkov, A.~A. (2013).
\newblock {\em Probability theory}.
\newblock Universitext. Springer, London, fifth edition.

\bibitem[Boucheron et~al., 2013]{boucheron13concentration}
Boucheron, S., Lugosi, G., and Massart, P. (2013).
\newblock {\em Concentration inequalities}.
\newblock Oxford University Press, Oxford.
\newblock A nonasymptotic theory of independence, With a foreword by Michel
  Ledoux.

\bibitem[Chebunin and Kovalevskii, 2016]{chebunin16functional}
Chebunin, M. and Kovalevskii, A. (2016).
\newblock Functional central limit theorems for certain statistics in an
  infinite urn scheme.
\newblock {\em Statist. Probab. Lett.}, 119:344--348.

\bibitem[Chebunin and Zuyev, 2022]{chebunin22functional}
Chebunin, M. and Zuyev, S. (2022).
\newblock Functional central limit theorems for occupancies and missing mass
  process in infinite urn models.
\newblock {\em J. Theoret. Probab.}, 35(1):1--19.

\bibitem[Darling, 1967]{darling67some}
Darling, D.~A. (1967).
\newblock Some limit theorems associated with multinomial trials.
\newblock In {\em Proc. {F}ifth {B}erkeley {S}ympos. {M}ath. {S}tatist. and
  {P}robability ({B}erkeley, {C}alif., 1965/66), {V}ol. {II}: {C}ontributions
  to {P}robability {T}heory, {P}art 1}, pages 345--350. Univ. California Press,
  Berkeley, CA.

\bibitem[Diaconis and Shahshahani, 1994]{diaconis94eigenvalues}
Diaconis, P. and Shahshahani, M. (1994).
\newblock On the eigenvalues of random matrices.
\newblock {\em J. Appl. Probab.}, 31A:49--62.
\newblock Studies in applied probability.

\bibitem[Durieu et~al., 2020]{durieu20infinite}
Durieu, O., Samorodnitsky, G., and Wang, Y. (2020).
\newblock From infinite urn schemes to self-similar stable processes.
\newblock {\em Stochastic Process. Appl.}, 130(4):2471--2487.

\bibitem[Durieu and Wang, 2016]{durieu16infinite}
Durieu, O. and Wang, Y. (2016).
\newblock From infinite urn schemes to decompositions of self-similar
  {G}aussian processes.
\newblock {\em Electron. J. Probab.}, 21:Paper No. 43, 23.

\bibitem[Fu and Wang, 2020]{fu20stable}
Fu, Z. and Wang, Y. (2020).
\newblock Stable processes with stationary increments parameterized by metric
  spaces.
\newblock {\em J. Theoret. Probab.}, 33(3):1737--1754.

\bibitem[Garza and Wang, 2024]{garza24limit}
Garza, J. and Wang, Y. (2024).
\newblock Limit theorems for random permutations induced by {C}hinese
  restaurant processes.
\newblock Arxiv preprint, \url{https://arxiv.org/abs/2412.02162}.

\bibitem[Gnedin et~al., 2007]{gnedin07notes}
Gnedin, A., Hansen, B., and Pitman, J. (2007).
\newblock Notes on the occupancy problem with infinitely many boxes: general
  asymptotics and power laws.
\newblock {\em Probab. Surv.}, 4:146--171.

\bibitem[Gnedin and Iksanov, 2020]{gnedin20nested}
Gnedin, A. and Iksanov, A. (2020).
\newblock On nested infinite occupancy scheme in random environment.
\newblock {\em Probab. Theory Related Fields}, 177(3-4):855--890.

\bibitem[Gr\"{u}bel and Kabluchko, 2016]{grubel16functional}
Gr\"{u}bel, R. and Kabluchko, Z. (2016).
\newblock A functional central limit theorem for branching random walks, almost
  sure weak convergence and applications to random trees.
\newblock {\em Ann. Appl. Probab.}, 26(6):3659--3698.

\bibitem[Hansen, 1990]{hansen90functional}
Hansen, J.~C. (1990).
\newblock A functional central limit theorem for the {E}wens sampling formula.
\newblock {\em J. Appl. Probab.}, 27(1):28--43.

\bibitem[Iksanov, 2016]{iksanov16renewal}
Iksanov, A. (2016).
\newblock {\em Renewal theory for perturbed random walks and similar
  processes}.
\newblock Probability and its Applications. Birkh\"auser/Springer, Cham.

\bibitem[Iksanov et~al., 2022]{iksanov22functional}
Iksanov, A., Kabluchko, Z., and Kotelnikova, V. (2022).
\newblock A functional limit theorem for nested {K}arlin's occupancy scheme
  generated by discrete {W}eibull-like distributions.
\newblock {\em J. Math. Anal. Appl.}, 507(2):Paper No. 125798, 24.

\bibitem[Iksanov and Kotelnikova, 2022]{iksanov22small}
Iksanov, A. and Kotelnikova, V. (2022).
\newblock Small counts in nested {K}arlin's occupancy scheme generated by
  discrete {W}eibull-like distributions.
\newblock {\em Stochastic Process. Appl.}, 153:283--320.

\bibitem[Karlin, 1967]{karlin67central}
Karlin, S. (1967).
\newblock Central limit theorems for certain infinite urn schemes.
\newblock {\em J. Math. Mech.}, 17:373--401.

\bibitem[Kingman, 1982]{kingman82coalescent}
Kingman, J. F.~C. (1982).
\newblock The coalescent.
\newblock {\em Stochastic Process. Appl.}, 13(3):235--248.

\bibitem[Merlev\`ede and Peligrad, 2013]{merlevede13rosenthal}
Merlev\`ede, F. and Peligrad, M. (2013).
\newblock Rosenthal-type inequalities for the maximum of partial sums of
  stationary processes and examples.
\newblock {\em Ann. Probab.}, 41(2):914--960.

\bibitem[Pipiras and Taqqu, 2017]{pipiras17long}
Pipiras, V. and Taqqu, M.~S. (2017).
\newblock {\em Long-range dependence and self-similarity}, volume~45 of {\em
  Cambridge Series in Statistical and Probabilistic Mathematics}.
\newblock Cambridge University Press.

\bibitem[Pitman, 2006]{pitman06combinatorial}
Pitman, J. (2006).
\newblock {\em Combinatorial stochastic processes}, volume 1875 of {\em Lecture
  Notes in Mathematics}.
\newblock Springer-Verlag, Berlin.
\newblock Lectures from the 32nd Summer School on Probability Theory held in
  Saint-Flour, July 7--24, 2002, With a foreword by Jean Picard.

\bibitem[Samorodnitsky, 2016]{samorodnitsky16stochastic}
Samorodnitsky, G. (2016).
\newblock {\em Stochastic processes and long range dependence}.
\newblock Springer, Cham, Switzerland.

\bibitem[Samorodnitsky and Taqqu, 1994]{samorodnitsky94stable}
Samorodnitsky, G. and Taqqu, M.~S. (1994).
\newblock {\em Stable non-{G}aussian random processes}.
\newblock Stochastic Modeling. Chapman \& Hall, New York.
\newblock Stochastic models with infinite variance.

\bibitem[Talagrand, 1991]{talagrand91new}
Talagrand, M. (1991).
\newblock A new isoperimetric inequality and the concentration of measure
  phenomenon.
\newblock In {\em Geometric aspects of functional analysis (1989--90)}, volume
  1469 of {\em Lecture Notes in Math.}, pages 94--124. Springer, Berlin.

\bibitem[Wieand, 2000]{wieand00eigenvalue}
Wieand, K. (2000).
\newblock Eigenvalue distributions of random permutation matrices.
\newblock {\em Ann. Probab.}, 28(4):1563--1587.

\end{thebibliography}
\end{document}